\documentclass[a4paper,12pt]{article}
\usepackage{latexsym}
\usepackage{amsmath,amsfonts,amssymb,amsthm}
\usepackage{graphicx}
\usepackage{comment,enumitem}
\usepackage{color}

\usepackage[backref=page]{hyperref}
\definecolor{modra3}{rgb}{0.1,0,0.7}
\hypersetup{colorlinks=true, linkcolor=modra3, urlcolor=modra3, citecolor=modra3, 
pdfpagemode=UseNone, pdfstartview=} 

\if10     
\usepackage[pagewise]{lineno}
\linenumbers
\fi

\newtheorem{theorem}{Theorem}[section]

\newtheorem{corollary}[theorem]{Corollary}
\newtheorem{lemma}[theorem]{Lemma}
\newtheorem{proposition}[theorem]{Proposition}
\newtheorem{conjecture}[theorem]{Conjecture}

\newtheorem{problem}[theorem]{Problem}
\newtheorem{fact}[theorem]{Fact}
\newtheorem{observation}[theorem]{Observation}

\theoremstyle{definition}

\newtheorem{example}[theorem]{Example}

\newcommand{\bbZ}{\mathbb{Z}}
\newcommand{\bbR}{\mathbb{R}} 

\newcommand{\cC}{{\mathfrak C}}
\newcommand{\chain}[2]{\cC(#1,#2)}  
\newcommand{\cS}{{\mathcal S}}  
\newcommand{\cP}{{\mathcal P}} 
\newcommand{\cPP}{\mathcal{PP}} 

\newcommand{\sNE}{\mathcal{NE}} 
\newcommand{\NE}{\mathrm{NE}} 
\newcommand{\sE}{\mathcal{E}} 
\newcommand{\E}{\mathrm{E}} 
\newcommand{\sEl}{\sE_\lambda(*,\pi_n)} 

\newcommand{\sEf}{\sE_f(*,\pi_n)}
\newcommand{\Sf}{S_f}
\newcommand{\SEf}{\mathcal{SE}_f(*,\pi_n)} 

\newcommand{\bu}{\texttt{-}} 
\newcommand{\ques}{\texttt{?}}
\newcommand{\ble}[1]{\mathtt{#1}} 
\newcommand{\seg}[1]{\underline{\ble{#1}}}
\newcommand{\pgap}{\texttt{*}}

\newcommand{\Tbb}{T_\ble{bb}}
\newcommand{\Btb}{B_\ble{tb}}
\newcommand{\Tbt}{T_\ble{bt}}

\DeclareMathOperator{\ides}{ides}
\DeclareMathOperator{\des}{des}
\DeclareMathOperator{\Img}{Img}

\begin{document}
\title{On the growth of the M\"{o}bius function of permutations\thanks{Supported by project 
16-01602Y of the Czech Science Foundation (GA\v{C}R). The third and fourth author were also 
supported by Charles University project UNCE/SCI/004.}}

\author{V\'it Jel\'inek
\thanks{Computer Science Institute, Charles University, Faculty of Mathematics and Physics, 
Malostransk\'e n\'am\v est\'i 25, Prague, Czech Republic, \texttt{jelinek@iuuk.mff.cuni.cz}}
\and
Ida Kantor
\thanks{Computer Science Institute, Charles University, Faculty of Mathematics and Physics, 
Malostransk\'e n\'am\v est\'i 25, Prague, Czech Republic, \texttt{ida@iuuk.mff.cuni.cz}}
\and
Jan Kyn\v{c}l
\thanks{Department of Applied Mathematics and Institute for Theoretical Computer Science, Charles 
University, Faculty of Mathematics and Physics, Malostransk\'e n\'am\v est\'i 25, Prague, Czech 
Republic, \texttt{kyncl@kam.mff.cuni.cz}}
\and
Martin Tancer
\thanks{Department of Applied Mathematics, Charles University, Faculty of Mathematics and Physics, 
Malostransk\'e n\'am\v est\'i 25, Prague, Czech Republic, \texttt{tancer@kam.mff.cuni.cz}}
} 

\thispagestyle{empty}
\maketitle

\begin{abstract} We study the values of the M\"obius function $\mu$ of 
intervals in the containment poset of permutations. 
We construct a sequence of 
permutations $\pi_n$ of size $2n-2$ for which $\mu(1,\pi_n)$ is given by a 
polynomial in $n$ of degree~7. This construction provides the fastest known 
growth of $|\mu(1,\pi)|$ in terms of $|\pi|$, improving a previous quadratic 
bound by Smith.

Our approach is based on a formula expressing the M\"obius 
function of an arbitrary permutation interval $[\alpha,\beta]$ in terms of the 
number of embeddings of the elements of the interval into~$\beta$.
\end{abstract}

\noindent\textbf{keywords}: M\"obius function, permutation poset, permutation embedding.

\section{Introduction}
The M\"obius function of a poset is a classical parameter with applications in 
combinatorics, number theory and topology. From the combinatorial perspective, 
an important problem is to study the M\"obius function of containment posets of 
basic combinatorial structures, such as 
words~\cite{BjornerSubword,BjornerFactor,McNamaraSagan}, integer 
compositions~\cite{SaganVatter,goyt}, integer partitions~\cite{Ziegler}, or set 
partitions~\cite{EdelmanSimion,EhrenborgReaddy}.

Wilf~\cite{Wilf} was the first to propose the study of the M\"obius function of 
the containment poset of permutations, and it quickly became clear that in its 
full generality this is a challenging topic. This is not too surprising, 
considering that it is computationally hard even to determine whether two 
permutations are comparable in the permutation poset~\cite{BBL}, and moreover, 
the poset of permutations is also hard to tackle by topological tools: for 
instance, most of its intervals are not shellable~\cite{McNamaraSt15}.

Thus, the known formulas for the M\"obius function of the permutation poset are 
restricted to permutations of specific structure, such as layered 
permutations~\cite{SaganVatter}, 132-avoiding permutations~\cite{ST}, separable 
permutations~\cite{BJJS}, or permutations with a fixed number of 
descents~\cite{Smith_one,Smith_descents}.

In this paper, we study the growth of the value $\max\{|\mu (1,\pi)|; |\pi|=n\}$ as a function of 
$n$. Here $\mu(1,\pi)$ is the M\"{o}bius function of the interval $[1,\pi]$, where $1$ is the unique 
permutation of size one; see Section~\ref{s:prelim} for precise definitions. We give a construction 
showing that the rate of growth of this value is $\Omega(n^7)$, improving 
a previous result by Smith~\cite{Smith_one}, who obtained a quadratic lower bound.

Specifically, for $n\ge 1$, we define the permutation 
$\pi_n\in\cS_{2n+2}$ by
\[\pi_n=n+1,1,n+3,2,n+4,3,n+5,\dots,n,2n+2,n+2. 
\]
See Figure~\ref{f:pi_345}. Our main result is the following formula for~$\mu(1, \pi_n)$.

\begin{figure}
\begin{center}
\includegraphics{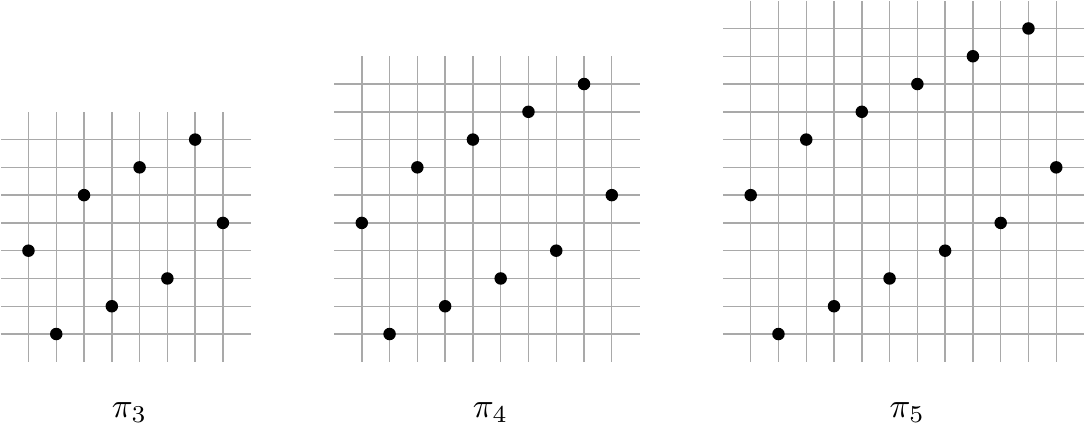}
\caption{Permutations $\pi_3$, $\pi_4$ and $\pi_5$.}
\label{f:pi_345}
\end{center}
\end{figure}

\begin{theorem}\label{thm-cross} For every $n\ge 2$, we have 
\[
\mu(1,\pi_n) = -\binom{n+2}{7}- \binom{n+1}{7}+ 2 \binom{n+2}{5}- \binom{n+2}{3} - \binom{n}{2}- 2n.
\]
\end{theorem}

The M\"obius function is closely related to the topological properties of the 
underlying poset. In particular, the M\"obius function is equal to the reduced 
Euler characteristic of the order complex of the poset, making it a homotopy 
invariant of the order complex. For an overview of the topological aspects of posets, 
the interested reader may consult the survey by Wachs~\cite{Wachs}.
Several previous results on the M\"obius 
function of posets are based on topological tools~\cite{SaganVatter,Smith_formula}.
In this paper, however, our approach is purely combinatorial and requires no 
topological background. 

Our main tool is a formula relating the M\"obius function of the permutation poset to 
the number of embeddings between pairs of permutations. We believe this formula 
(Proposition~\ref{pro-form} and the closely related Corollary~\ref{cor-rek}) may find further 
applications in the study of the M\"obius function. In fact, several such applications already 
emerged from our joint work with Brignall and Marchant~\cite{Zeros}, which is being
prepared for publication in parallel with this paper.

\section{Definitions and preliminaries}\label{sec-defs}
\label{s:prelim}

\paragraph*{Permutations and their diagrams} 
Let $[n]$ denote the set 
$\{1,2,\dots,n\}$. A \emph{permutation} of size $n$ 
is a bijection $\pi$ of $[n]$ onto itself. We represent such a permutation 
$\pi$ as the sequence of values $\pi(1),\pi(2),\dots,\pi(n)$. 
If there is no risk of ambiguity, we omit the commas and write, for example, $312$ 
for the permutation $\pi$ with $\pi(1)=3$, $\pi(2)=1$ and $\pi(3)=2$. 

The \emph{diagram} of a permutation $\pi$ is the set of points $\{(i,\pi(i));\ i\in[n]\}$ in the
plane; in other words, it is the graph of $\pi$ as a function. Let $\cS_n$ be the set of permutations of 
size $n$, and let $\cS=\bigcup_{n\ge 1} \cS_n$ be the set of all finite 
permutations.

\paragraph*{Embeddings} 
A sequence of real numbers $a_1,a_2,\dots,a_n$ is \emph{order-isomorphic} to a sequence
$b_1,b_2,\dots,b_n$ if for every $i,j\in[n]$ we have $a_i<a_j \Leftrightarrow b_i<b_j$. 

An \emph{embedding} of a permutation $\sigma\in\cS_k$ into a permutation 
$\pi\in\cS_n$ is a function $f\colon[k]\to[n]$ such that 
$f(1)<f(2)<\dots<f(k)$, and the sequence $\pi(f(1)),\pi(f(2)),\dots,\pi(f(k))$ 
is order-isomorphic to $\sigma(1),\sigma(2),\dots,\sigma(k)$.
The \emph{image} of an embedding $f$ is the set $\Img(f)=\{f(i);\; 
i\in[k]\}$. Observe that for a given $\pi$, the set $\Img(f)$ determines both 
$f$ and $\sigma$ uniquely.
 
If there is an embedding of $\sigma$ into $\pi$, we say that $\pi$ 
\emph{contains}~$\sigma$, and write $\sigma\le\pi$, otherwise we say that $\pi$ 
\emph{avoids}~$\sigma$. The containment relation $\le$ is a partial 
order on~$\cS$. We will call the pair $(\cS,\le)$ the \emph{permutation poset}.

We let $\sE(\sigma,\pi)$ denote the set of embeddings of $\sigma$ into $\pi$, and we let 
$\E(\sigma,\pi)$ denote the cardinality of~$\sE(\sigma,\pi)$.

\begin{figure}
  \centerline{\includegraphics[width=0.8\linewidth]{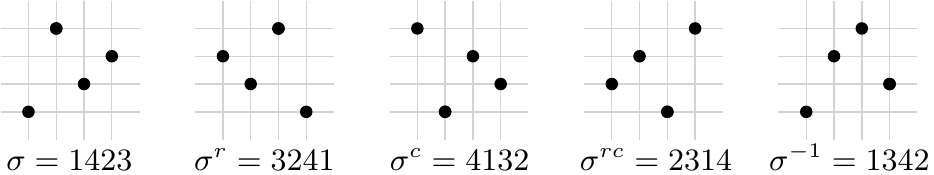}}
  \caption{The main permutation symmetries.}
  \label{fig-sym}
\end{figure}

\paragraph*{Permutation symmetries} For a permutation $\pi=\pi(1)\pi(2)\dots\pi(n)$, 
its \emph{reverse} $\pi^r$ is the permutation $\pi(n)\pi(n-1)\dots\pi(1)$, its 
\emph{complement} $\pi^c$ is the permutation 
$n+1-\pi(1),n+1-\pi(2),\dots,n+1-\pi(n)$, its \emph{reverse-complement} 
$\pi^{rc}$ is the permutation $(\pi^r)^c$, and its \emph{inverse} is the 
permutation $\pi^{-1}\in\cS_n$ with the property $\pi^{-1}(\pi(i))=i$ for every 
$i\in[n]$. Observe that these operations correspond to reflections or rotations
of the diagram of $\pi$; see Figure~\ref{fig-sym}. Although reflections and rotations generate 
an 8-element group of symmetries, in this paper we only need the five symmetries 
depicted in Figure~\ref{fig-sym}.

Note that these operations are poset automorphisms of $(\cS, \le)$, that is, 
$\sigma\le\pi$ is equivalent to $\sigma^r\le \pi^r$, $\sigma^c\le\pi^c$, 
$\sigma^{rc}\le\pi^{rc}$, and $\sigma^{-1}\le\pi^{-1}$.

\paragraph*{The M\"obius function}
For a poset $(P,\le)$, we let $[x,y]$ denote the 
closed interval $\{z\in P;\; x\le z\le y\}$, and $[x,y)$ the half-open 
interval $\{z\in P;\; x\le z<y\}$. A poset $(P,\le)$ is \emph{locally finite} if
each of its intervals is finite. Given a locally finite poset $(P,\le)$, we  
define its \emph{M\"obius function} $\mu\colon P\times P\to \bbZ$ by the 
recurrences
\[
 \mu(x,y)=\begin{cases}
             0 \text{ if } x \not\le y\\
1 \text{ if } x=y\\
-\sum_{z\in[x,y)} \mu(x,z) \text{ if } x<y.
            \end{cases}
\]

A \emph{chain} from $x\in P$ to $y\in P$ is a set 
$C=\{x_0,x_1,\dots,x_k\}\subseteq P$ such that $x_0=x$, $x_k=y$, and 
$x_{i-1}<x_i$ for every $i\in[k]$. The \emph{length} of a chain $C$, denoted by
$\ell(C)$, is defined as $|C|-1$. We let $\chain{x}{y}$ denote the set of 
all chains from $x$ to~$y$.

For an arbitrary set $\cC$ of chains, we define the \emph{weight of $\cC$}, 
denoted by $w(\cC)$, as $\sum_{C\in\cC} (-1)^{\ell(C)}$.

We will need a classical identity known as Philip Hall's Theorem, which expresses the M\"obius function of an interval as the reduced Euler characteristic of the corresponding order complex. For details, see, for example, Stanley~\cite[Proposition 3.8.5]{Stan} or Wachs~\cite[Proposition 1.2.6]{Wachs}.

\begin{fact}[Philip Hall's Theorem]
\label{fac-hall} 
If $(P,\le)$ is a locally finite poset with elements $x$ and $y$, then 
$\mu(x,y)=w(\chain{x}{y})$.
\end{fact}

The following symmetry property of the M\"obius function is a direct consequence of Philip Hall's Theorem.

\begin{corollary}\label{cor-dual}
Let $(P,\le)$ be a locally finite poset with M\"obius function~$\mu$. Let 
$\le^*$ be the partial order on $P$ defined by $x\le^* y \Leftrightarrow y\le x$. The 
M\"obius function $\mu^*$ of the poset $(P,\le^*)$ then satisfies 
$\mu^*(x,y)=\mu(y,x)$.
\end{corollary}

We will often use the following easy identities, where the first one follows 
from the definition of $\mu$, and the second one from Corollary~\ref{cor-dual}.
 
\begin{fact}\label{fac-sum} For a locally finite poset $P$ and a 
pair of elements $x,y\in P$ with $x<y$, we have $\sum_{z\in[x,y]} \mu(x,z)=0$  
and $\sum_{z\in[x,y]} \mu(z,y)=0$.
\end{fact}

We will also use the M\"obius inversion formula, which is a basic property of the M\"obius 
function. The following form of the formula can be deduced, for example, from Proposition 3.7.2 in 
Stanley's book~\cite{Stan}.

\begin{fact}[M\"obius inversion formula]\label{fac-mif} Let $P$ be a locally finite poset with maximum 
element $y$, let $\mu$ be the M\"obius function of $P$, and let $f\colon 
P\to\bbR$ be a function. If a function $g\colon P\to\bbR$ is defined by
\[
g(x)=\sum_{z\in[x,y]} f(z),
\]
then for every $x\in P$, we have
\[
f(x)=\sum_{z\in[x,y]} \mu(x,z)g(z).
\]
\end{fact}

\paragraph*{A lemma on decreasing patterns}
From now on, we only deal with the poset $(\cS, \le)$ of permutations ordered by 
the containment relation, and $\mu$ refers to the M\"obius function of 
this poset.

\begin{lemma}\label{lem-dec} Let $\delta_k$ be the decreasing permutation 
of size $k$; that is, $\delta_k=k,(k-1),\dots,1$. For any permutation $\pi$ 
other than $1$ or $12$, we 
have
\begin{equation}\label{eq-dec}
 \sum_{k=1}^{|\pi|} \mu(\delta_k,\pi) = 0.
\end{equation}
\end{lemma}

\begin{proof} Consider the set of chains $\cC=\bigcup_{k=1}^{|\pi|} 
\chain{\delta_k}{\pi}$. In view of Fact~\ref{fac-hall}, equation~\eqref{eq-dec} 
is equivalent to $w(\cC)=0$.

Define two sets of chains $\cC_1$ and $\cC_2$ as follows:
\begin{align*}
 \cC_1&=\{C\in\cC;\; C\text{ contains a decreasing permutation of 
size at least 2} \}, \text{ and}\\
\cC_2&=\cC\setminus\cC_1.
\end{align*}
Clearly, $w(\cC)=w(\cC_1)+w(\cC_2)$. We will show that both $\cC_1$ and $\cC_2$ 
have zero weight.

To see that $w(\cC_1)=0$, consider a parity-exchanging involution $\Phi_1$ on 
$\cC_1$ defined as follows: if $C\in\cC_1$ contains the permutation 1, define 
$\Phi_1(C)=C\setminus\{1\}$, otherwise define $\Phi_1(C)=C\cup\{1\}$. We see 
that $\Phi_1$ is an involution on $\cC_1$ that maps chains of odd length to 
chains of even length and vice versa. Therefore $w(\cC_1)=0$.

To deal with $\cC_2$, consider the mapping $\Phi_2$ that maps a chain 
$C\in\cC_2$ to $C\setminus\{12\}$ if $C$ contains $12$, and it maps $C$ to 
$C\cup\{12\}$ otherwise. This is again easily seen to be a parity-exchanging 
involution on $\cC_2$, showing that $w(\cC_2)=0$.
\end{proof}

In our applications, we will use Lemma~\ref{lem-dec} in the situation 
when $\pi$ avoids $321$. In such cases, the sum on the left-hand side of 
\eqref{eq-dec} has at most two nonzero summands, and the identity can be 
rephrased as follows.

\begin{corollary}\label{cor-dec}
 Any 321-avoiding permutation $\pi$ other than 1 or 12 satisfies 
\[\mu(1,\pi)=-\mu(21,\pi).\]
\end{corollary}

We remark that a slightly more restricted case of Corollary~\ref{cor-dec} has 
already been proven by Smith~\cite[Lemma 3.6]{Smith_descents}, by a topological 
argument.

\paragraph*{M\"obius function via embeddings} 

The core of our argument is the following general formula expressing the M\"obius function in terms 
of another function~$f$. It can be seen as a version of the M\"obius inversion formula. It can be generalized in a straightforward way to an arbitrary locally finite poset, although for our purposes, we only state it in the permutation setting.

\begin{proposition}\label{pro-form}
Let $\sigma$ and $\pi$ be arbitrary permutations, and let $f\colon [\sigma,\pi]\to\bbR$ be a 
function satisfying $f(\pi)=1$. We then have
\begin{equation}
 \mu(\sigma,\pi)= f(\sigma) - \sum_{\lambda\in 
[\sigma,\pi)} 
\mu(\sigma,\lambda)\sum_{\tau\in[\lambda,\pi]} f(\tau).
\label{eq-form}
\end{equation}
\end{proposition}

\begin{proof}
Fix $\sigma$, $\pi$ and $f$. For $\lambda\in[\sigma,\pi]$, define
$g(\lambda)=\sum_{\tau\in[\lambda,\pi]} 
f(\tau)$. Using Fact~\ref{fac-mif} for the poset $P=[\sigma,\pi]$, we obtain
\begin{equation}
f(\sigma)=\sum_{\lambda\in[\sigma,\pi]} \mu(\sigma,\lambda)g (\lambda).
\label{eq-f-pomoci-g}
\end{equation}
Substituting the definition of $g(\lambda)$ into the identity~\eqref{eq-f-pomoci-g} and 
using the assumption $f(\pi)=1$, we get
\begin{align*}
f(\sigma)&=\sum_{\lambda\in[\sigma,\pi]}\mu(\sigma,\lambda)
\sum_{\tau\in[\lambda,\pi]} f(\tau)\\
&=\mu(\sigma,\pi) + \sum_{\lambda\in[\sigma,\pi)}\mu(\sigma,\lambda)
 \sum_{\tau\in[\lambda,\pi]} f(\tau), 
\end{align*}
from which the proposition follows.
\end{proof}

In our applications of Proposition~\ref{pro-form}, we shall always use the function $f$ defined as 
$f(\tau)=(-1)^{|\tau|-|\pi|} \E(\tau,\pi)$, where $\pi$ is assumed to be fixed. We state this 
special case of Proposition~\ref{pro-form} as a corollary.

\begin{corollary}\label{cor-rek}
For any two permutations $\sigma$ and $\pi$, we have
\[
 \mu(\sigma,\pi)= (-1)^{|\pi|-|\sigma|}\E(\sigma,\pi) - \sum_{\lambda\in 
[\sigma,\pi)} 
\mu(\sigma,\lambda)\sum_{\tau\in[\lambda,\pi]} 
(-1)^{|\pi|-|\tau|}\E(\tau,\pi).
\]
\end{corollary}

We remark that the formula of Corollary~\ref{cor-rek} has a similar structure to another summation
formula for the M\"obius function derived previously by Smith \cite[Theorem 
19]{Smith_formula}.


\paragraph*{Descents and inverse descents}
The \emph{inverse descent} of a permutation 
$\pi=\pi(1)\pi(2)\dots\pi(m)$ 
is a pair of indices $i,j\in[m]$ such that $i<j$ and $\pi(i)=\pi(j)+1$. 
Let $\ides(\pi)$ be the number of inverse descents of~$\pi$. For 
example, $315264$ has two inverse descents, corresponding to $(i,j)=(1,4)$ 
and $(i,j)=(3, 6)$. Observe that if $\sigma$ is contained in $\pi$, then  
$\ides(\sigma)\le\ides(\pi)$. 

The inverse descent statistic is closely related to the more familiar 
descent statistic, where a \emph{descent} in a permutation $\pi$ is a pair 
of indices $i,j$ such that $\pi(i)>\pi(j)$ and $j=i+1$. The number of descents 
of $\pi$ is denoted by $\des(\pi)$. Note that $\des(\pi)=\ides(\pi^{-1})$.

Suppose that $\pi$ has only one inverse descent, occurring at positions $i<j$ 
with $\pi(i)=\pi(j)+1$. We say that an element $\pi(k)$ is a \emph{top element} 
if $\pi(k)\ge \pi(i)$, and it is a \emph{bottom element} if $\pi(k)\le \pi(j)$. 
We also say that $k$ is a \emph{top position} of $\pi$ if $\pi(k)$ is a top 
element, and bottom positions are defined analogously. Note that each element of 
$\pi$ is either a top element or a bottom element, and that the top elements, as 
well as the bottom elements, form an increasing subsequence of~$\pi$.

By replacing each top element of $\pi$ by the symbol `$\ble{t}$' and each 
bottom element by the symbol `$\ble{b}$', we encode a permutation $\pi$ with 
$\ides(\pi)=1$ into a word $w(\pi)$ over the alphabet $\{\ble{b},\ble{t}\}$. For 
example, for $\pi=31245$ we get $w(\pi)=\ble{tbbtt}$. Note that $w(\pi)$ 
determines $\pi$ uniquely. On the other hand, some words over the alphabet 
$\{\ble{b},\ble{t}\}$ do not correspond to any permutation $\pi$ with 
$\ides(\pi) = 1$; for example, the word $\ble{bbtt}$, and in general, every word where all symbols `$\ble{b}$' appear before all symbols `$\ble{t}$'.

This encoding of permutations into words was introduced by 
Smith~\cite{Smith_descents}, who also generalized it to permutations with $k$ 
inverse descents\footnote{Actually, Smith works with descents rather than 
inverse descents, but since the inverse operation is a poset automorphism of 
$(\cS, \le)$, this change does not affect the relevant results.}, by encoding them into 
words over an alphabet of size~$k+1$. The key feature of Smith's encoding is 
that if $\sigma$ and $\pi$ have the same number of inverse descents, then 
$\sigma\le\pi$ if and only if $w(\sigma)$ is a subword of $w(\pi)$, that is, the 
word $w(\sigma)$ forms a (not necessarily consecutive) subsequence of~$w(\pi)$. 
In other words, if $\ides(\sigma)=\ides(\pi)$ then the interval $[\sigma,\pi]$ 
is isomorphic, as a poset, to the interval $[w(\sigma),w(\pi)]$ in the subword 
order.

To express the M\"obius function in the subword order, Bj\"orner~\cite{BjornerSubword,BjornerFactor} 
has introduced the notion of normal embeddings among words. This notion was adapted by 
Smith~\cite{Smith_descents} to the permutation setting, to express the M\"obius function of 
permutations with a fixed number of descents. We will present Smith's definition of normal 
embeddings below. Let us remark that other authors have used different notions of normal 
embeddings, suitable for other special types of 
permutations~\cite{Zeros,SaganVatter,BJJS,Smith_formula}. 

Let $\pi=\pi(1)\pi(2)\dotsc\pi(n)$ be a permutation. An \emph{adjacency} in $\pi$ is a maximal 
consecutive sequence $\pi(i)\pi(i+1)\dotsc\pi(i+k)$ satisfying $\pi(i+j)=\pi(i)+j$ for each 
$j=1,\dotsc,k$; in other words, it is a maximal sequence of consecutive increasing values at 
consecutive positions in~$\pi$. An embedding $f$ of a permutation $\sigma$ into $\pi$ is 
\emph{normal} if for each adjacency $\pi(i)\pi(i+1)\dotsc\pi(i+k)$ of $\pi$, the positions 
$i+1,\dots,i+k$ all belong to~$\Img(f)$. Let $\sNE(\sigma,\pi)$ be the set of normal embeddings of 
$\sigma$ into~$\pi$, and let $\NE(\sigma,\pi)$ be the cardinality of $\sNE(\sigma,\pi)$. 

As an example, consider the permutations $\sigma=123$ and $\pi=165234$. There are four embeddings of 
$\sigma$ into $\pi$, with images $\{1,4,5\}$, $\{1,4,6\}$, $\{1,5,6\}$ and $\{4,5,6\}$. The 
permutation $\pi$ has one adjacency of length more than $1$, namely the sequence $234$ at positions 
$4$, $5$ and~$6$. Thus, an embedding $f$ into $\pi$ is normal if $\Img(f)$ contains both $5$ 
and~$6$. In particular, $\NE(\sigma,\pi)=2$.

Note that if all the adjacencies in $\pi$ have length $1$, then every embedding of a permutation 
$\sigma$ into $\pi$ is normal.

Using Bj\"orner's formula for the M\"obius function of the subword 
order~\cite{BjornerSubword,BjornerFactor}, Smith~\cite{Smith_descents} obtained 
the following result.

\begin{fact}[Smith \protect{\cite[Proposition 
3.3]{Smith_descents}}]
\label{fac-descents}
If $\sigma$ and $\pi$ satisfy $\des(\sigma)=\des(\pi)$, then 
$\mu(\sigma,\pi)=(-1)^{|\pi|-|\sigma|}\NE(\sigma,\pi)$.
\end{fact}

Observing that $\NE(\sigma^{-1},\pi^{-1})=\NE(\sigma,\pi)$, and recalling that 
$\mu(\sigma^{-1},\pi^{-1})=\mu(\sigma,\pi)$ and $\ides(\pi^{-1})=\des(\pi)$, we can rephrase 
Fact~\ref{fac-descents} as follows.

\begin{corollary}
\label{cor_descents} 
If $\sigma$ and $\pi$ satisfy 
$\ides(\sigma)=\ides(\pi)$, then 
\[
\mu(\sigma,\pi)=(-1)^{|\pi|-|\sigma|}\NE(\sigma,\pi).
\]
\end{corollary}

Let $\pi\in\cS_{2n}$ be the permutation with one inverse descent and encoding 
$w(\pi)=\ble{tbtbtb\dots tb}$. By Corollary~\ref{cor_descents} and 
Corollary~\ref{cor-dec}, 
\[\mu(1,\pi)=-\mu(21,\pi)=-\NE(21,\pi)=-\binom{n+1}{2}.\] This example, pointed 
out by Smith~\cite{Smith_one}, gave the largest previously known growth of
$|\mu(1,\pi)|$ in terms of $|\pi|$. We note that Brignall and Marchant~\cite{BM} 
have recently found another, substantially different example of a family of 
permutations $\pi$ for which they conjecture that $|\mu(1,\pi)|$ is quadratic 
in~$|\pi|$.

\begin{figure}
  \centerline{\includegraphics[width=0.4\linewidth]{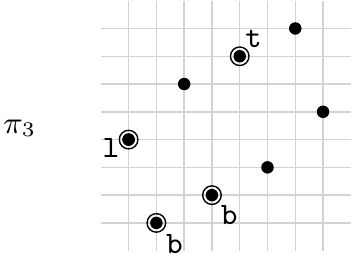}}
  \caption{The permutation $\pi_3$. The circled elements correspond to the 
embedding $f= \ble{lb\bu bt\bu\bu\bu}$ of the permutation $3124$.}
  \label{fig-pi}
\end{figure}

\section{Proof of Theorem~\ref{thm-cross}}\label{sec-big}

In this section, we present a proof of Theorem~\ref{thm-cross}. We 
will assume throughout that $n$ is a fixed integer greater than 1.

Recall that we defined the permutation 
$\pi_n\in\cS_{2n+2}$ by
\[\pi_n=n+1,1,n+3,2,n+4,3,n+5,\dots,n,2n+2,n+2. 
\]
See Figures~\ref{f:pi_345} and~\ref{fig-pi}. Note that by transposing the two values $n+1$ and 
$n+2$ in $\pi_n$, we would 
obtain the permutation $\pi\in\cS_{2n+2}$ with $w(\pi)=\ble{tbtbt\dots tb}$, 
considered by Smith. Note that $\ides(\pi_n)=2$ for any $n\ge 1$, with the two inverse descents 
defined by the pairs of positions $(1,2n)$ and $(3,2n+2)$.

We will refer to the leftmost element of $\pi_n$, that is, the value $n+1$, as the 
\emph{left element}, the rightmost one as the \emph{right element}, the 
elements $1,2,\dots,n$ as the \emph{bottom elements}, and 
$n+3,n+4,\dots,2n+2$ as the \emph{top elements}. If $\pi_n(i)$ is a bottom 
element, we say that $i$ is a \emph{bottom position} of $\pi_n$, and similarly 
for left, right and top positions. There is a close link between the left, right, top and bottom 
elements of $\pi_n$, and the top and bottom elements of any permutation $\sigma\le\pi_n$ with 
$\ides(\sigma)=1$, as we shall see in Lemma~\ref{lem-proper}.

We use the letters $\ble{t,b,l,r}$ to represent the top, bottom, left and right 
elements, respectively. By replacing each element of $\pi_n$ by the 
corresponding letter, we obtain \emph{the encoding} of $\pi_n$; for example, the 
encoding of $\pi_3$ is $\ble{lbtbtbtr}$.

Note that any subword of the encoding of $\pi_n$ uniquely determines the 
subpermutation formed by the corresponding elements of~$\pi_n$. For example, the 
subsequence $\ble{lbbt}$ corresponds to the subpermutation $3124$. However, a 
permutation $\sigma\le\pi_n$ may correspond to several different words: for 
example, $41253$ corresponds to either $\ble{lbbtb}$ or $\ble{tbbtb}$ or 
$\ble{tbbtr}$. In 
particular, an interval $[\sigma,\pi_n]$ will not in general be 
isomorphic to any interval in the subword order, and we cannot use the results
of Bj\"orner and Smith directly to obtain a formula for $\mu(\sigma,\pi_n)$. On 
the positive side, if $[\sigma,\pi_n]$ is not isomorphic to an interval in the 
subword order, then $|\mu(\sigma,\pi_n)|$ can be much larger than 
$\NE(\sigma,\pi_n)$, as witnessed by Theorem~\ref{thm-cross}.

In the rest of this section, we only consider embeddings of permutations into $\pi_n$, unless we explicitly specify otherwise.
For our convenience, we will often represent an embedding $f$ into $\pi_n$ by 
a word, using the letters $\ble{b,t,l,r}$ for the bottom, top, left and right 
positions in $\Img(f)$, and the symbol $\bu$ for positions not in 
$\Img(f)$. For example, $f = \ble{lb\bu b t \bu \bu \bu}$ is an 
embedding of $3124$ into $\pi_3=41627385$ using the first, second, fourth and 
fifth elements (in other words, $\Img(f)=\{1,2,4,5\}$; see Figure~\ref{fig-pi}). 
We call this the {\em hyphen-letter encoding} of an embedding.

\subsection{Outline of the proof of Theorem~\ref{thm-cross}}

Clearly, $\pi_n$ avoids $321$, and hence $\mu(1,\pi_n)=-\mu(21,\pi_n)$ by 
Corollary~\ref{cor-dec}. We therefore focus on finding the value of 
$\mu(21,\pi_n)$. By Corollary~\ref{cor-rek}, $\mu(21,\pi_n)$ equals
\[
(-1)^{|\pi_n|-|21|}\E(21,\pi_n) - \sum_{\lambda\in 
[21,\pi_n)} 
\mu(21,\lambda)\sum_{\tau\in[\lambda,\pi_n]} 
(-1)^{|\pi_n|-|\tau|}\E(\tau,\pi_n),
\]
which, since $\pi_n$ has even size, simplifies to
\begin{equation}\label{eq-mf}
\mu(21,\pi_n)=\E(21,\pi_n) - \sum_{\lambda\in 
[21,\pi_n)} 
\mu(21,\lambda)\sum_{\tau\in[\lambda,\pi_n]} 
(-1)^{|\tau|}\E(\tau,\pi_n).
\end{equation}

We say that a permutation $\lambda\in[21,\pi_n)$ is \emph{vanishing} if 
the expression 
\[\mu(21,\lambda)\sum_{\tau\in[\lambda,\pi_n]} (-1)^{|\tau|}\E(\tau,\pi_n),\] 
which is the outer summand on the right-hand side of~\eqref{eq-mf}, is 
equal to zero. To simplify equation~\eqref{eq-mf}, we will first 
establish several sufficient conditions for $\lambda$ to be vanishing; in 
particular, it will turn out that there are only polynomially many 
non-vanishing $\lambda$, and we can describe their structure explicitly. 

Our next concern will be to express, for a fixed non-vanishing 
$\lambda\in[21,\pi_n)$, the value
\[
S_\lambda=\sum_{\tau\in[\lambda,\pi_n]} (-1)^{|\tau|}\E(\tau,\pi_n),
\]
that is, the value of the inner sum in equation~\eqref{eq-mf}.

Recall that $\sE(\tau,\pi_n)$ is the set of embeddings of
  $\tau$ into $\pi_n$. We further let 
\[
\sEl=\bigcup_{\tau\in[\lambda,\pi_n]} \sE(\tau,\pi_n).
\]

For an embedding $g$, we let $|g|$ denote the size of the 
permutation being embedded; equivalently, $|g|=|\Img(g)|$. With this notation, 
$S_\lambda$ can be written as follows:
\[
S_\lambda = \sum_{g\in\sEl} (-1)^{|g|}.
\]
Call an embedding 
$g$ \emph{odd} if $|g|$ is odd, and even otherwise. To find a formula for 
$S_\lambda$, we will find a partial matching on $\sEl$ between odd and even embeddings, thereby 
cancelling out their contribution to~$S_\lambda$. The unmatched embeddings will have a very specific 
structure, and we will be able to count them exactly.

Let $g$ be an embedding of a permutation $\tau$ into $\pi_n$, and let 
$i\in[2n+2]$ be an index corresponding to a position in~$\pi_n$.
The \emph{$i$-switch} of the embedding~$g$, denoted by $\Delta_i(g)$, is the 
embedding uniquely determined by the following properties:
\begin{align*}
  \Img(\Delta_i(g))&= \Img(g)\cup\{i\} \text{ if } i\not\in\Img(g)\text{, and}\\
  \Img(\Delta_i(g))&= \Img(g)\setminus\{i\} \text{ if } i\in\Img(g).
\end{align*}
Note that for any $i\in[2n+2]$, the $i$-switch is an involution on 
the set of embeddings into $\pi_n$, that is, $\Delta_i(\Delta_i(g))=g$ for 
any~$g$. Note also that $\Delta_i$ is parity-exchanging, that is, it maps even 
embeddings to odd ones and vice versa. Switches will be our main tool to obtain
cancellations between even and odd embeddings contributing to $S_\lambda$ for a 
fixed~$\lambda$. The idea of using switches to get parity-exchanging involutions
on a set of embeddings is quite common in the literature, and can be traced back 
at least to Bj\"orner's work on the subword order~\cite{BjornerFactor}.

\subsection{Vanishing lambdas} 

We will now identify
sufficient conditions for $\lambda$ to be vanishing, that is, for 
$\mu(21,\lambda)S_\lambda$ to be equal to zero.

\begin{lemma}\label{lem-ides}
 Any permutation $\lambda\in[21,\pi_n)$ with $\ides(\lambda)=2$ is vanishing.
\end{lemma}

\begin{proof}
 Fix $\lambda\in[21,\pi_n)$ with $\ides(\lambda)=2$. Since $\ides(\pi_n)=2$, it 
follows that any $\tau\in[\lambda,\pi_n]$ has $\ides(\tau)=2$. In particular, 
for any such $\tau$ we have $(-1)^{|\tau|}\NE(\tau,\pi_n) = \mu(\tau,\pi_n)$ by 
Corollary~\ref{cor_descents}. Since $\pi_n$ has no adjacency of length more than $1$, all the 
embeddings into $\pi_n$ are normal, and in particular $\NE(\tau,\pi_n)=\E(\tau,\pi_n)$.
Therefore,
\begin{align*}
 S_\lambda&=\sum_{\tau\in[\lambda,\pi_n]} (-1)^{|\tau|}\E(\tau,\pi_n)\\
&=\sum_{\tau\in[\lambda,\pi_n]} \mu(\tau,\pi_n)\\
&=0 &\text{ by Fact~\ref{fac-sum}, }
\end{align*}
showing that $\lambda$ is vanishing.
\end{proof}

Note that any $\lambda$ containing 21 has at least one inverse descent. 
Lemma~\ref{lem-ides} therefore implies that any non-vanishing $\lambda$ in 
$[21,\pi_n)$ satisfies $\ides(\lambda)=1$. From now on, we focus on 
permutations $\lambda$ with one inverse descent.

We say that a permutation $\lambda\in[21,\pi_n)$ \emph{omits} a position 
$i\in[2n+2]$ if there is no embedding $f$ of $\lambda$ into $\pi_n$ such that 
$i\in\Img(f)$.

\begin{lemma}\label{lem-omit}
 If a permutation $\lambda\in[21,\pi_n)$ omits a position $i\in[2n+2]$, then 
$\lambda$ is vanishing. 
\end{lemma}

\begin{proof}
 We easily see that the $i$-switch $\Delta_i$ is a parity-exchanging involution 
on the set $\sEl$, and therefore $S_\lambda=0$.
\end{proof}

\begin{corollary}
\label{cor_vanish}
Let $\lambda$ be a permutation of size $m$ with one inverse 
descent. Then $\lambda$ is vanishing whenever it satisfies at least one of the 
following conditions.
\begin{itemize}
\item[a)] The two leftmost elements $\lambda(1)$ and $\lambda(2)$ are both top 
elements.
\item[b)] The leftmost element $\lambda(1)$ is a bottom element, 
and $\lambda$ has at least one other bottom element $\lambda(i)$ with
$1<i<m$.
\item[c)] The two rightmost elements $\lambda(m-1)$ and $\lambda(m)$ are both 
bottom elements.
\item[d)] The rightmost element $\lambda(m)$ is a top element, 
and $\lambda$ has at least one other top element $\lambda(i)$ with 
$1<i<m$.
\item[e)] The permutation $\lambda$ has only one top element and at least 
three bottom elements.
\item[f)] The permutation $\lambda$ has only one bottom element and at least 
three top elements.
\end{itemize}
\end{corollary}

\begin{proof}
If Case a) occurs, then $\lambda$ omits the position $i=2$. To see this, note that only $\lambda(1)$ 
or $\lambda(2)$ can be embedded to $\pi_n(2)$; however, since $\lambda(1)$ and $\lambda(2)$ are both 
top elements, $\lambda$ contains a bottom element $\lambda(i)$ that is smaller than both 
$\lambda(1)$ and $\lambda(2)$. Since $\pi_n$ contains no element smaller than $\pi_n(2)$, there is no embedding of $\lambda$ into $\pi_n$ that maps a top element to $\pi_n(2)$. Hence 
$\lambda$ omits the position 2.

If Case b) occurs, we claim that $\lambda$ omits the position~$1$. Indeed, if there were an 
embedding $f$ of $\lambda$ into $\pi_n$ with $1\in\Img(f)$, then necessarily $f$ would map 
$\lambda(1)$ to $\pi_n(1)$. Since $\lambda(1)$ is a bottom element, it is the smallest element of 
$\lambda$, and hence $f$ cannot map any element of $\lambda$ to an element of $\pi_n$ smaller than 
$\pi_n(1)$. In particular, $f$ maps no element of $\lambda$ to a bottom element of $\pi_n$. 
By assumption, $\lambda$ has a bottom element $\lambda(i)$ with $1<i<m$. This element must be mapped 
to a top element of~$\pi_n$. All the elements of $\lambda$ to the right of $\lambda(i)$ are larger 
than $\lambda(i)$, and therefore no element of $\lambda$ can be mapped to the rightmost element 
of~$\pi_n$. It follows that $f$ maps all the elements of lambda to the left and top elements of 
$\pi_n$, showing that $\lambda$ is an increasing permutation, contradicting $\ides(\lambda)=1$.

Cases c) and d) are symmetric to Cases a) and b) via the reverse-com\-ple\-ment symmetry (note 
that $\pi_n^{rc}=\pi_n$, and consequently if $\lambda$ is vanishing then so is $\lambda^{rc}$). If 
Case e) occurs, then at least one of b) and c) occurs as well. Case f) is again symmetric to e).
\end{proof}

We say that a permutation $\lambda$ of size $m$ with one inverse 
descent is a \emph{cup} if $\lambda(1)$ and $\lambda(m)$ are top elements and 
the 
remaining elements are bottom elements; in other words, $\lambda$ is the 
permutation $(m-1),1,2,\dots,(m-2),m$. We say that $\lambda$ is a \emph{cap} 
if $\lambda(1)$ and $\lambda(m)$ are bottom elements and the remaining elements 
are top elements, that is, $\lambda=1,3,4,\dots,m,2$. 

Suppose that $\lambda\le \tau \le \pi_n$, $f$ is an embedding of $\lambda$
and $g$ is an embedding of~$\tau$. We say that $f$ is {\em compatible} 
with $g$ and also that $g$ is {\em compatible} 
with $f$ if $\Img(f)$ is a subset of~$\Img(g)$. 

\begin{lemma}\label{lem-cup}
 If $\lambda$ is a cup or a cap of size $m\ge 3$, then $\lambda$ is vanishing.
\end{lemma}

\begin{proof}
Let $\lambda$ be a cup. Recall that we are considering embeddings into the 
permutation $\pi_n$ with $n>1$, that is, $\pi_n$ has size $2n+2\ge 6$.
We say that an embedding $f$ of $\lambda$ into $\pi_n$ is \emph{broad} if 
$f(m)=2n+2$, and $f$ is \emph{narrow} otherwise. Observe that 
any broad embedding must satisfy~$f(1)=1$. We partition the set $\sEl$ 
into two subsets $A$ and $B$, where $A$ is the set of those embeddings 
$g\in\sEl$ that are compatible with at least one narrow embedding of 
$\lambda$, while $B$ contains those embeddings $g\in\sEl$ that are only 
compatible with broad embeddings of~$\lambda$. Note that $\Delta_{2n+2}$ is an 
involution on the set $A$, showing that $A$ does not contribute to~$S_\lambda$. 

We now deal with the set~$B$. Consider first the situation when $m\ge 4$. We 
claim that in such case $\Delta_3$ is an involution on the set~$B$. To see this, 
observe first that for any $g\in B$ we have $\Delta_3(g)\in\sEl$, since any 
broad embedding of $\lambda$ compatible with $g$ is also compatible 
with~$\Delta_3(g)$. It remains to show that there is no narrow embedding of 
$\lambda$ compatible with $\Delta_3(g)$. Indeed, if $f$ were such a narrow 
embedding, we would necessarily have $f(1)=3$, and therefore $f(m)<2n+2$, since 
$\pi_n(f(1))$ must be smaller than $\pi_n(f(m))$. But then, if we redefine the 
value of $f(1)$  from 3 to 1, we obtain a narrow embedding of $\lambda$ 
compatible with $g$, which is impossible since $g$ is in~$B$. Thus, $B$ does 
not contribute to $S_\lambda$ either, and $\lambda$ is vanishing.

Now suppose $\lambda$ has size 3, that is, $\lambda=213$. Consider an embedding 
$g\in B$. Note that both $1$ and $2n+2$ are in $\Img(g)$, and in the 
hyphen-letter encoding of $g$, all occurrences of the symbol $\ble{b}$ must be 
to the right of any occurrence of $\ble{t}$, otherwise $g$ would be compatible 
with a narrow embedding of~$\lambda$. Let $B'\subseteq B$ be the set of those 
embeddings $g\in B$ whose hyphen-letter encoding has only one symbol $\ble{b}$ 
and this symbol appears at position $2n$, and let $B''$ be the set~$B\setminus 
B'$. We note that $\Delta_{2n}$ is an involution on $B''$, while $\Delta_3$ is 
an involution on~$B'$ (notice that here we require that $n\neq 1$). We conclude 
that any cup $\lambda$ is vanishing.

A cap is the reverse-complement of a cup, and therefore caps are vanishing as 
well, by symmetry.
\end{proof}

So far we have identified several cases when $\lambda$ is vanishing because 
$S_\lambda$ is zero. We now focus on the situations when $\mu(21,\lambda)=0$, 
which also implies that $\lambda$ is vanishing. For a permutation 
$\lambda$ with $\ides(\lambda)=1$, we define a 
\emph{top repetition} to be a pair $(\lambda(i),\lambda(i+1))$ of two 
consecutive top elements in $\lambda$, and similarly a \emph{bottom repetition} 
to be a pair of two consecutive bottom elements.

By Corollary~\ref{cor_descents}, $|\mu(21,\lambda)|=\NE(21,\lambda)$ whenever 
$\ides(\lambda)=1$. Observing that a normal embedding of 21 into $\lambda$ must 
contain the right element of any repetition in its image, we reach the 
following conclusion.

\begin{observation}\label{obs-mu} 
Let $\lambda$ be a permutation with 
$\ides(\lambda)=1$. If $\lambda$ has at least two top repetitions, or at least 
two bottom repetitions, or a top repetition appearing to the right of a bottom 
repetition, then $21$ has no normal embedding into $\lambda$ and consequently 
$\mu(21,\lambda)=0$. In particular, such $\lambda$ is vanishing.
\end{observation}

\subsection{Proper lambdas} 

We will say that a permutation $\lambda\in[21,\pi_n)$ of size $m$
is \emph{proper} if it satisfies the following three conditions:
\label{def-proper}
\begin{enumerate}
\item[1)] $\ides(\lambda)=1$,
\item[2)] $\lambda(1)$ and $\lambda(m-1)$ are top elements, while $\lambda(2)$ and 
$\lambda(m)$ are bottom elements (for $m=2$ the elements of each pair coincide and $\lambda=21$),
\item[3)] $\lambda$ has at most one top repetition and at most one bottom 
repetition; moreover, if it has both a top repetition and a bottom repetition, 
then the top repetition is to the left of the bottom repetition.
\end{enumerate}
Condition 1) and Corollary~\ref{cor_descents} imply the following identity.

\begin{corollary}
\label{cor_proper_lambda}
For every proper permutation $\lambda$, we have 
\[
\mu(21,\lambda)=(-1)^{|\lambda|}\NE(21,\lambda).
\]
\end{corollary}

Let~$\cP_n\subseteq[21,\pi_n)$ be the set of proper permutations, and let 
$\cP_{n,m}$ be the set of proper permutations of size~$m$. By Lemma~\ref{lem-ides}, Corollary~\ref{cor_vanish}, Lemma~\ref{lem-cup} and Observation~\ref{obs-mu}, any non-vanishing permutation $\lambda\in[21,\pi_n)$ is 
proper. In particular, we may simplify identity~\eqref{eq-mf} as follows:
\begin{equation}\label{eq-proper}
 \mu(21,\pi_n)=\E(21,\pi_n) - \sum_{\lambda\in 
\cP_n} 
\mu(21,\lambda)S_\lambda.
\end{equation}

Note that $21$ is the smallest proper permutation, and that there are no proper 
permutations of size~3. For future reference, we state several easy facts about 
embeddings of proper permutations.

\begin{lemma}\label{lem-properrc}
If $\lambda$ is a proper permutation, then its reverse-complement $\lambda^{rc}$ is proper as well.
Moreover, we have $\mu(21,\lambda)=\mu(21,\lambda^{rc})$ and $S_\lambda=S_{\lambda^{rc}}$.
\end{lemma}
\begin{proof}
The fact that $\lambda^{rc}$ is proper follows directly from the definition of proper permutation.
The identity $21^{rc}=21$ and the fact that the reverse-complement operation is 
an automorphism of the permutation poset imply that the 
intervals $[21,\lambda]$ and $[21,\lambda^{rc}]$ are isomorphic as posets, and hence 
$\mu(21,\lambda)=\mu(21,\lambda^{rc})$. It remains to prove that $S_\lambda=S_{\lambda^{rc}}$. Recall 
that $\pi^{rc}_n=\pi_n$. Thus, for any permutation $\tau$ we have 
$\E(\tau,\pi_n)=\E(\tau^{rc},\pi_n)$. Moreover, $\tau$ belongs to 
$[\lambda,\pi_n]$ if and only if $\tau^{rc}$ belongs to $[\lambda^{rc},\pi_n]$. Together, this gives
\[
S_\lambda = \sum_{\tau\in[\lambda,\pi_n]} (-1)^{|\tau|} \E(\tau,\pi_n)= 
\sum_{\tau^{rc}\in[\lambda^{rc},\pi_n]} (-1)^{|\tau^{rc}|} 
\E(\tau^{rc},\pi_n)=S_{\lambda^{rc}},
\]
as claimed.
\end{proof}

\begin{lemma}\label{lem-proper}
Let $\lambda$ be a proper permutation of size~$m$, and let $f\colon[m]\to[2n+2]$ 
be a function. Then $f$ is an embedding of $\lambda$ into $\pi_n$ if and only if 
it satisfies the following three conditions:
\begin{enumerate}
\item[{\rm 1)}] The function $f$ is strictly increasing, that is, $f(i)<f(i+1)$ for every 
$i\in[m-1]$.
\item[{\rm 2)}] If $i\in[m]$ is a top position of $\lambda$, then $f(i)$ is a 
left or top position of~$\pi_n$, and if $i$ is a bottom position of $\lambda$ 
then $f(i)$ is a bottom or right position of~$\pi_n$.
\item[{\rm 3)}] At most one of the two values $1$ and $2n+2$ is in~$\Img(f)$.
\end{enumerate}
\end{lemma}

\begin{proof}
Suppose $f$ is an embedding of~$\lambda$. Then $f$ is strictly increasing by 
definition. Moreover, each top element of $\lambda$ has a smaller bottom element 
to its right, and therefore it has to be mapped to an element of $\pi_n$ that 
has a smaller element of~$\pi_n$ to its right; that is, it has to be mapped to a left or top 
element. Symmetrically, each bottom element of~$\lambda$ must be mapped to a bottom or right element 
of~$\pi_n$. Finally, to see that $\Img(f)$ cannot contain both $1$ and 
$2n+2$, recall that for any $\lambda\in\cP_{n,m}$ we have $\lambda(1)>\lambda(m)$.
Therefore, every embedding of $\lambda$ satisfies the three properties of the 
lemma. Conversely, it is easy to observe that any function satisfying the three 
properties is an embedding of~$\lambda$.
\end{proof}

As a direct consequence of Lemma~\ref{lem-proper}, we get the following result.

\begin{corollary}\label{cor-max}
Let $f$ and $f'$ be two embeddings of a proper permutation $\lambda$ 
into~$\pi_n$. Define a function $f^*$ by $f^*(i)=\max\{f(i),f'(i)\}$. Then 
$f^*$ is also an embedding of~$\lambda$.
\end{corollary}

We call the function $f^*$ defined in Corollary~\ref{cor-max} the \emph{pointwise maximum} of $f$ 
and $f'$.
We remark that Corollary~\ref{cor-max} does not generalize to improper 
permutations. Consider for instance $\lambda=3124$, which is a cup and 
therefore not proper, being embedded into $\pi_3=41627385$. 
Take the two embeddings $f=\ble{l \bu \bu b \bu b 
\bu r}$, with $\Img(f)=\{1,4,6,8\}$, and $f'=\ble{\bu \bu t b \bu b t 
\bu}$, with $\Img(f')=\{3,4,6,7\}$. Their pointwise maximum $f^*=\ble{\bu 
\bu t b \bu b \bu r}$ is not an embedding of $\lambda=3124$ but of 
$4123$.

\subsection{\texorpdfstring{Determining $S_\lambda$ for a proper $\lambda$}{Determining S\_lambda for a proper lambda}}

Fix a proper permutation~$\lambda$. Our goal now is to determine 
the value $S_\lambda=\sum_{g\in\sEl} (-1)^{|g|}$. To this end, we will 
describe cancellations between odd and even embeddings in $\sEl$, so that the 
value of $S_\lambda$ can be determined by a small and well-structured subset of 
uncancelled embeddings.

Let $<_L$ denote the \emph{lexicographic order} on the set 
$\sE(\lambda,\pi_n)$, which is a total order defined as follows. Let $f$ and 
$f'$ be two embeddings of $\lambda$ into $\pi_n$, and let $i$ be the smallest 
index for which $f(i)\neq f'(i)$. If $f'(i)< f(i)$, then put $f'<_{L} f$.
If $g\in\sEl$ is an embedding, then $<_L$ can be restricted to a total order on 
the set of embeddings of $\lambda$ that are compatible with~$g$. The maximum
element in this ordered set is called the {\em rightmost embedding of $\lambda$ 
compatible with~$g$}, or just the \emph{rightmost embedding of $\lambda$ 
in~$g$}. Corollary~\ref{cor-max} implies the following fact.

\begin{corollary}\label{cor_pointwise}
The rightmost embedding of $\lambda$ in $g$ is 
the pointwise maximum of all the embeddings of $\lambda$ compatible with~$g$.
\end{corollary}

The notion of rightmost embedding will serve us to establish cancellations between odd and even 
embeddings in~$\sEl$. We remark that a similar approach has been used by 
Bj\"orner~\cite{BjornerFactor} in the subword poset (who uses the term \emph{final embedding} for 
rightmost embedding) as well as by Sagan and Vatter~\cite{SaganVatter}.

We now show that rightmost embeddings can be constructed by a natural greedy right-to-left 
procedure. (Alternatively, they could also be characterized as `locally rightmost', in the sense 
that no element can be shifted to the right alone.) Let $\lambda$ be a proper permutation of size 
$m$, and let $g\in\sEl$ be an embedding. We say that an embedding $f$ of $\lambda$ is \emph{greedy 
in $g$} if $f$ is constructed by the following rules: 
\begin{itemize}
 \item $f(m)$ is equal to the largest (that is, rightmost) bottom or right 
position in~$\Img(g)$, 
\item for each top position $i\in[m-1]$ of $\lambda$, assuming $f(i+1)$ has 
already been defined, $f(i)$ is equal to the largest left or top position 
$j\in\Img(g)$ such that $j<f(i+1)$, and similarly,
\item for each bottom position $i\in[m-1]$ of $\lambda$, assuming $f(i+1)$ has 
already been defined, $f(i)$ is equal to the largest bottom position 
$j\in\Img(g)$ such that $j<f(i+1)$.
\end{itemize}

We say that an embedding $f$ of $\lambda$ is \emph{almost greedy in~$g$}
if $\Img(g)$ contains the rightmost position $2n+2$, and $f$ is greedy 
in the embedding $g^-$ defined by $\Img(g^-)=\Img(g)\setminus\{2n+2\}$.

\begin{lemma}\label{lem-greedy}
For any proper permutation $\lambda$ and any $g\in\sEl$, the rightmost 
embedding of $\lambda$ in $g$ is greedy or almost greedy in~$g$. 
Moreover, if the rightmost embedding is almost greedy, then every embedding 
$f'$ of $\lambda$ into $g$ satisfies $1\in\Img(f')$ and therefore 
$2n+2\not\in\Img(f')$, and there is no greedy embedding of $\lambda$ in~$g$.
\end{lemma}
\begin{proof}
Let $m$ be the size of $\lambda$, and let $f$ be the rightmost embedding of 
$\lambda$ in~$g$. Suppose that $f$ is not greedy in $g$, and let 
$i$ be the largest index for which $f(i)$ differs from the value prescribed by 
the definition of the greedy embedding.

First consider the case $i<m$. Since $f(i)$ differs from its 
greedy value, there must be a position $j\in\Img(g)$ such that $f(i)<j<f(i+1)$, 
and either both $j$ and $f(i)$ are bottom positions, or $j$ is a top position 
and $f(i)$ is a top or left position. In any case, we can define a new 
embedding $f^+$ by
\[
 f^+(x)=
\begin{cases}
  f(x) \text{ for } x\neq i,\\
  j \text{ for } x=i.
\end{cases}
\]
By Lemma~\ref{lem-proper}, $f^+$ is an embedding of $\lambda$, and it is clearly 
compatible with~$g$. However, we have $f<_L f^+$, contradicting the choice 
of~$f$.

Suppose now that $i=m$, that is,  
the rightmost bottom or right position $j\in\Img(g)$ is greater 
than $f(m)$. By defining $f^+$ as in the previous paragraph, we again get 
contradiction, except for the case when $f(1)=1$ and $j=2n+2$. In such 
situation $f$ is almost greedy. Furthermore, since $f(1)=1$, there 
can be no embedding $f'$ of $\lambda$ into $g$ with $f'(1)>1$, because the 
pointwise maximum of $f$ and $f'$ would then contradict the choice of~$f$. 
Since any embedding $f'$ of $\lambda$ compatible with $g$ satisfies $f'(1)=1$, 
we see that $2n+2\not\in\Img(f')$ by Lemma~\ref{lem-proper}, and therefore $f'$ 
is not greedy.
\end{proof}

For $f\in\sE(\lambda,\pi_n)$, let $\sEf$ be the set of all the embeddings 
$g\in\sEl$ such that $f$ is the rightmost embedding of $\lambda$ in~$g$. Let 
\[
\Sf=\sum_{g\in\sEf} (-1)^{|g|}. 
\]
In particular, we have 
\begin{align*}
  \sEl &=\bigcup_{f\in\sE(\lambda,\pi_n)} \sEf, \text{ and}\\
S_\lambda&=\sum_{f\in \sE(\lambda,\pi_n)} \Sf.
\end{align*}

\paragraph*{Proper pairs} 
We will now show that $\Sf=0$ except when $f$ has 
a specific form.

For $f\in\sE(\lambda,\pi_n)$, a \emph{gap} in $f$ is an open interval 
$(f(i),f(i+1))$ of integers such that $f(i+1)>f(i)+1$. We say that the gap has 
\emph{type $\ble{tb}$}, or that it is a $\ble{tb}$-gap, if $\lambda(i)$ is a 
top element and $\lambda(i+1)$ a bottom element; types $\ble{bt}$, $\ble{bb}$ 
and $\ble{tt}$ are defined analogously. For instance, the embedding 
$f=\ble{lbt\bu\bu b\bu btb\bu\bu}$ has two gaps, namely the 
$\ble{tb}$-gap $(f(3),f(4))=\{4,5\}$ and the $\ble{bb}$-gap $(f(4),f(5))=\{7\}$.

Note that a top repetition or a bottom repetition in $\lambda$ will necessarily 
form a gap of type $\ble{tt}$ or $\ble{bb}$, respectively, in any embedding 
of~$\lambda$ in $\pi_n$.

\begin{lemma}\label{lem-sf}
Let $\lambda$ be a proper permutation and let $f$ be an embedding of $\lambda$ 
into~$\pi_n$ with $\Sf\neq0$. Then $f$ satisfies the following three conditions:
\begin{enumerate}
\item[{\rm 1)}] Position $1$ is in $\Img(f)$.
\item[{\rm 2)}] If $f$ has a $\ble{tb}$-gap $(f(i), f(i+1))$, then $i\ge 4$, 
and $\lambda(i-1)$ and $\lambda(i-2)$ are both bottom elements.
\item[{\rm 3)}] If $f$ has a $\ble{bt}$-gap $(f(i), f(i+1))$, then $i\ge 4$ and 
either $\lambda(i-1)$ and $\lambda(i-2)$ are both top elements, or 
$\lambda(i-1)$ is a bottom element and $\lambda(i-2)$ and $\lambda(i-3)$ 
are both top elements.
\end{enumerate}
\end{lemma}

\begin{proof}
Let $m$ be the size of $\lambda$.
To prove Part 1), we claim that if $1\not\in\Img(f)$, then $\Delta_1$ is an 
involution on $\sEf$, and hence $\Sf=0$. To see this, choose $g\in\sEf$ and define 
$g'=\Delta_1(g)$. Clearly $g'$ is compatible with $f$, so to prove that $g'$ is in $\sEf$, we only 
need to argue that $f$ is the rightmost embedding of $\lambda$ compatible with~$g'$. To see this, 
choose an embedding $f'\in\sE(\lambda,\pi_n)$ compatible with $g$, and note that if
$1\in\Img(f')$ then $f'<_L f$, and if $1\not\in\Img(f')$ then $f'$ is also compatible 
with $g$ and hence $f'\le_L f$.

Now we prove Part 2). Let $(f(i),f(i+1))$ be a $\ble{tb}$-gap, and 
let $j=f(i)$ and $k=f(i+1)$. We will show that if $\lambda(i-1)$ and $\lambda(i-2)$ are not both 
bottom elements, then $\Delta_{j+1}$ is an involution on~$\sEf$ and hence~$\Sf=0$. To 
see this, suppose that $\Delta_{j+1}$ is not such an involution, that is, there is 
an embedding $g\in\sEf$ such that $\Delta_{j+1}(g)$ is not in~$\sEf$. Let $g'=\Delta_{j+1}(g)$. 
Since $j+1$ is not in $\Img(f)$, $f$ is compatible with~$g'$. As 
$g'\not\in\sEf$, $g'$ is compatible with an embedding of $\lambda$ greater 
than~$f$ in the $<_L$-order. In particular, $\Img(g')=\Img(g)\cup\{j+1\}$. Let 
$f'$ be the rightmost embedding of $\lambda$ 
in~$g'$. We have $f<_L f'$, and also $f(\ell)\le f'(\ell)$ for every $\ell\in[m]$ by 
Corollary~\ref{cor_pointwise}. 

\begin{table}
\begin{tabular}{lcccc}
 & $f$ & $g$ & $f'$ & $g'$ \\
After the definition of $f'$: & $\ble{t\bu\bu\bu\bu b}$ & $\ble{t\bu\ques\ques\ques b}$ & $\ble{\ques\ques\ques\ques\ques\ques}$  & $\ble{tb\ques\ques\ques b}$ \\
No $\ble{t}$ in the gap in $g'$: & $\ble{t\bu\bu\bu\bu b}$ &  $\ble{t\bu\bu\ques\bu b}$ & $\ble{\ques\ques\bu\ques\bu\ques}$ & $\ble{tb\bu\ques\bu b}$ \\
$\ble{b}$ on position $j+1$ in $f'$: & $\ble{t\bu\bu\bu\bu b}$ & $\ble{t\bu\bu\ques\bu b}$ & $\ble{\ques b\bu\ques\bu\ques}$ & $\ble{tb\bu\ques\bu b}$ \\
$\ble{b}$ on position $k$ in $f'$: & $\ble{t\bu\bu\bu\bu b}$ & $\ble{t\bu\bu\ques\bu b}$ & $\ble{\ques b\bu\ques\bu b}$ & $\ble{tb\bu\ques\bu b}$ \\
at most one $\ble{bb}$-gap in $f'$: & $\ble{t\bu\bu\bu\bu b}$ & $\ble{t\bu\bu\bu\bu b}$ & $\ble{\ques b\bu\bu\bu b}$ & $\ble{tb\bu\bu\bu b}$ \\
\end{tabular}
\caption{Evolution of the conditions on $f$, $g$, $f'$ and $g'$ inside a $\ble{tb}$ gap of $f$. (For sake of example, $k = j + 5$.)}
\label{t:evolution}
\end{table}

We observe that $\Img(g)$ contains no top position $j'$ in the gap 
$(j,k)$, otherwise $f$ would would not be rightmost in $g$, since it could be modified to map $i$ to $j'$ instead of~$j$. 
Therefore $\Img(g')$ contains no such top position either. Follow Table~\ref{t:evolution} for steps in this paragraph.
Since $f'$ is compatible with $g'$ but not with $g$, $\Img(f')$ contains $j+1$. Also, $\Img(f')$ contains~$k$, 
otherwise we could shift the rightmost $\ble{b}$ in $f'$ inside the $\ble{tb}$-gap to the right, contradicting the choice of $f'$.
Since $j+1$ and $k$ are both bottom positions of $\pi_n$, and $\Img(f')$ 
has no top position in the gap, we conclude that $\lambda$ has a bottom 
repetition mapped to positions $j+1$ and $k$ by~$f'$. This bottom 
repetition must appear to the left of the element $\lambda(i)$, because $f<_L f'$ and since $j+1$ 
is not in $\Img(f)$ and $j$ is a top position of $\pi_n$, $f$ must map the two elements of the 
repetition strictly to the left of $f(i)=j$.

It remains to show that the bottom repetition of $\lambda$ appears at positions 
$i-2$ and~$i-1$. Suppose that the bottom repetition appears at positions $i'$ 
and $i'+1$ for some $i'<i-2$. By Condition 3 of the definition of proper permutation 
(page~\pageref{def-proper}), the positions $i'+2,i'+3,\dots,m$ do not have 
any repetition in $\lambda$, that is, they correspond to alternating top and 
bottom elements, starting with a top one. Moreover, $i>i'+2$ by assumption, 
therefore in fact $i\ge i'+4$, since $i'+2$ and $i$ are both 
top positions of $\lambda$. Note that $f'(i'+2)>k$, since $f'(i'+1)=k$.

We define a mapping $f^+\colon[m]\to[2n+2]$, contradicting the choice of $f$, 
as follows: 
\[
f^+(x)=
\begin{cases}
  f(x) \text{ for } x< i,\\
  f'(x-2) \text{ for } x\ge i.
\end{cases}
\]
We easily verify that $f^+$ is an 
embedding of $\lambda$ using Lemma~\ref{lem-proper}: Condition~2) follows 
directly from the definition of $f^+$, Condition 1) follows from 
$f(i-1)<f(i)=j<k=f'(i'+1)<f'(i-2)$, and Condition 3) follows from 
$f^+(m)=f'(m-2)<f'(m)\le 2n+2$. Moreover, $f^+$ is compatible with $g$ since 
$j+1 \notin \Img(f^+)$, and 
$f^+(i)>f(i)$, contradicting the choice of~$f$. This proves Part 2) of the 
lemma.

The proof of Part 3) is similar. Let $(f(i), f(i+1))$ be a 
$\ble{bt}$-gap, and let $j=f(i)$ and $k=f(i+1)$. We will again show that 
$\Delta_{j+1}$ is an involution on $\sEf$, unless $\lambda$ satisfies the 
conditions of Part 3). 

Let $g\in\sEf$ be again an embedding such that $\Delta_{j+1}(g)$ is not in $\sEf$. Let 
$g'=\Delta_{j+1}(g)$ and let $f'$ be the rightmost embedding of $\lambda$ in~$g'$. 
For the same reason as in Part 2), we have $\Img(g')=\Img(g)\cup\{j+1\}$, $f<_L 
f'$, and $f(\ell)\le f'(\ell)$ for every $\ell\in[m]$.
No bottom position in the $\ble{bt}$-gap can belong to $\Img(g)$, otherwise $f$ 
would not be rightmost in~$g$. Again, for the same reason as in Part 2), both 
$j+1$ and $k$ are in $\Img(f')$. Thus, $\lambda$ contains a top repetition to the 
left of $\lambda(i)$, at positions $i_1$ and $i_1+1$ for some $i_1<i$, such that $f'(i_1)=j+1$ and $f'(i_1+1)=k$. 
Let $i_2$ be the largest top position of $\lambda$ smaller than 
$i$. In particular, $i_2$ is equal to $i-1$ or $i-2$. We need to prove that $i_2=i_1+1$, which is equivalent to the condition in part~3).

Suppose that $i_2>i_1+1$. This implies $i_2\ge i_1+3$, since 
$i_1+2$ is a bottom position in~$\lambda$. Since $\lambda$ is 
proper, there is no repetition among the elements 
$\lambda(1),\lambda(2),\dots,\lambda(i_1)$, that is, these elements form an alternation of 
top and bottom elements, starting with a top one. 

We define a mapping $f^+\colon[m]\to[2n+2]$, contradicting the choice of $f$, 
as follows:
\[
 f^+(x)=\begin{cases}
        f(x+2) \text{ for }x\le i_1-2\\
	f(i_2-1) \text{ for } x= i_1-1\\
	f(i_2)\text{ for } x=i_1\\
	f'(x)\text{ for } x>i_1
       \end{cases}
\]
We verify that $f^+$ is an embedding of $\lambda$ using Lemma~\ref{lem-proper}: 
Condition 2) follows directly from the definition of $f^+$ and from the `alternating property' of 
$\lambda$, Condition 1) 
follows from $f(i_1)<f(i_2-1)$ and $f(i_2)<f(i)=j<f'(i_1)<f'(i_1+1)$, and 
Condition 3) follows from $f^+(1)=f(3)>1$. Moreover, $f^+$ is compatible with 
$g$ since $j+1 \notin \Img(f^+)$, and 
$f<_L f^+$, contradicting the definition of~$f$.
\end{proof}

We say that $(\lambda,f)$ is a \emph{proper pair} if $\lambda$ is a proper 
permutation and $f$ is an embedding of $\lambda$ into $\pi_n$ that satisfies 
the three conditions of Lemma~\ref{lem-sf}. Let $\cPP_n$ be the set of all 
proper pairs $(\lambda,f)$ where $f$ is an embedding into~$\pi_n$. Combining formula~\eqref{eq-proper}, Corollary~\ref{cor_proper_lambda} and Lemma~\ref{lem-sf},  
we get

\begin{equation}\label{eq-pp}
 \mu(21,\pi_n)=\E(21,\pi_n) - \sum_{(\lambda,f)\in 
\cPP_n} 
(-1)^{|\lambda|}\NE(21,\lambda)\Sf.
\end{equation}

\paragraph*{Singular embeddings} 
Our goal is now to compute, for a proper pair $(\lambda,f)$, the value 
$\Sf=\sum_{g\in\sEf} (-1)^{|g|}$. To this end, we will introduce further 
cancellations on the set~$\sEf$. Let $j\in[2n+2]$ be the smallest index not belonging to~$\Img(f)$. 
For every embedding $g\in\sEf$, let $g'=\Delta_{j}(g)$.
Clearly, $g' \in \sEl$, since 
$f$ is compatible with~$g'$. However, $g'$ is not necessarily in 
$\sEf$, because $g'$ may be compatible with another embedding $f'$ of 
$\lambda$ with $f<_L f'$.

\begin{example}
Let $\lambda=3142$ and $g=\ble{lbtb \bu \bu 
\bu \bu \bu r}$. The rightmost embedding of $\lambda$ in $g$ is $f=\ble{lbtb\bu \bu \bu \bu \bu 
\bu}$, hence $g$ is in~$\sEf$. The first position not in $\Img(f)$ is the 
fifth one and we have $g'=\ble{lbtbt\bu \bu \bu \bu r}$, where 
the rightmost embedding of $\lambda$ is $\ble{\bu \bu tbt \bu \bu 
\bu \bu r}$. Hence $g'\not\in\sEf$.
\end{example}

We say that an embedding $g\in\sEf$ is \emph{$f$-regular} if $g'\in \sEf$; otherwise we say that 
$g$ is \emph{$f$-singular}. Let $\SEf$ be the set of $f$-singular embeddings in~$\sEf$.

Observe that if $g$ is $f$-regular then $g'$ is also 
$f$-regular. Thus, the $j$-switch restricts 
to a parity-exchanging involution on the set of $f$-regular 
embeddings. This shows that the contributions of $f$-regular embeddings 
to~$\Sf$ cancel out, and therefore 
\begin{equation}\label{eq_Sf_singular}
\Sf=\sum_{g\in\SEf} (-1)^{|g|}.
\end{equation}

We will now analyze $f$-singular embeddings in detail.

For an embedding $f$ of a permutation $\lambda$ of size $m$ 
into $\pi_n$, the set $\{i\in[2n+2];\; i>f(m)\}$ is called the \emph{tail} 
of~$f$. A \emph{segment} of $f$ is a maximal subset of consecutive integers belonging to~$\Img(f)$. Thus,
if $1\in\Img(f)$, then the set $[2n+2]$ can be partitioned into segments, gaps and the tail of~$f$. 

\begin{lemma}
\label{lem-singular} 
Let $(\lambda,f)$ be a proper pair where $|\lambda|=m$ and $f$ is an embedding into $\pi_n$. Let $j=\min([2n+2]\setminus \Img(f))$. Let $g$ be an $f$-singular embedding and let $g'=\Delta_j(g)$. Let $f'$ be 
the rightmost embedding of $\lambda$ in~$g'$. Then the following properties 
hold:
\begin{enumerate}[label=\rm{(\alph*)}]
 \item \label{it1} $\Img(g')=\Img(g) \cup \{j\}$, $j\notin \Img(g)$, $j\in \Img(f')$, and $f <_L 
f'$.
 \item \label{it2} For every $i \in [m]$ we have $f(i)<f'(i)$. 
 \item \label{it3} The embedding $f$ is almost greedy in $g$, and 
$f'$ is greedy in~$g'$. In particular, there is no greedy embedding of~$\lambda$ in $g$, and
$\Img(f')$ and $\Img(g)$ both contain the position $2n+2$. 
 \item \label{it4} $\Img(g)$ has no bottom position in the tail of~$f$, no top 
position in any $\ble{tt}$-gap or $\ble{tb}$-gap of $f$, and no bottom position 
in any $\ble{bb}$-gap or $\ble{bt}$-gap of~$f$. 
 \item \label{it5} If $f$ has at least one gap, then $\Img(g)$ has no position in 
the leftmost gap of~$f$, and has at least one top position in the tail of~$f$.
 \item \label{it6} If $f$ has at least two gaps, and the second gap from the left 
is a $\ble{bt}$-gap or a $\ble{tb}$-gap, then $\Img(g)$ has no position in the 
second gap.
\item \label{it7} If $f$ has a $\ble{tt}$-gap $(f(i),f(i+1))$ and a 
$\ble{bb}$-gap $(f(i+2),\allowbreak f(i+3))$, then $\Img(g)$ contains at most 
one top position in the $\ble{bb}$-gap. 
\end{enumerate}
\end{lemma}

\begin{proof} 
\ref{it1} These facts directly follow from $g$ being $f$-singular.

\ref{it2} By Corollary~\ref{cor_pointwise}, we have $f(i)\le f'(i)$ for every~$i$. Since the 
interval $[1,j-1]$ is a segment in $f$, the leftmost $j-1$ elements of $\lambda$ form an 
alternation of top and bottom elements starting with a top one, and $f(i)=i$ for every $i<j$.  
Since $j$ is in $\Img(f')\setminus \Img(f)$, we have $j=f'(i')$ for some $i'< j-1$ that has the same parity as $j$.
Since $f'$ is rightmost, by Lemma~\ref{lem-greedy} we have $f'(i)=i+(j-i')$ for every 
$i\le i'$. Consequently, $f(i)<f'(i)$ for every $i\le j-1$. 

Suppose that for some $i_0\ge j$ we have $f(i_0)=f'(i_0)$. Define a mapping $f^+$ by
\[
 f^+(x)=\begin{cases} f'(x) \text{ if } j-1\le x<i_0\\ f(x)\text{ 
    otherwise.}\end{cases}
\]
Clearly, $f^+$ satisfies all three conditions of Lemma~\ref{lem-proper} and thus 
it is an embedding of $\lambda$. Also $j\notin \Img(f^+)$, so $f^+$ is 
compatible with $g$. Finally, $f(j-1)=j-1$ and $f'(j-1)>j$ imply $f<_L f^+$; 
this is a contradiction with $f$ being rightmost in $g$.

\ref{it3} By Lemma~\ref{lem-greedy}, we know that $f$ is greedy or almost greedy in 
$g$, and $f'$ is greedy or almost greedy in~$g'$. Note that the value 
$j\in\Img(g')\setminus\Img(g)$ cannot be equal to either of $f(m)$ or $f'(m)$: indeed, 
we have either $j<m$ (in case $f$ has a gap), or $j=m+1$ (when $f$ has no gap) 
and in the latter case $j$ is a top position and $m$ a bottom one. In particular, $f'(m)\in 
\Img(g)$. By Part \ref{it2}, we have $f(m)<f'(m)$, which implies that $f$ is almost greedy in $g$, 
further implying that $f'(m)=2n+2$ and $f'$ 
is greedy in~$g'$.
By Lemma~\ref{lem-greedy}, this implies that there is no greedy embedding of~$\lambda$ in $g$.

\ref{it4} If $i\in\Img(g)$ is a bottom position in the tail of $f$, or a top 
position in a $\ble{tt}$-gap or $\ble{tb}$-gap of $f$, or a bottom position in 
a $\ble{bb}$-gap or $\ble{bt}$-gap of~$f$, we get a contradiction with the 
almost-greedy property of $f$, since the position of $\lambda$ mapped by $f$ to the largest position of $\Img(f)$ to the 
left of $i$ would be mapped to $i$ or to the right of $i$ by an almost-greedy embedding.

\ref{it5} Suppose that $\Img(g)$ has a position in the leftmost 
gap $(f(j-1),f(j))$ of $f$, and let $k$ be the leftmost such position. Assume 
that $\lambda(j-1)$ is a top element; the other case is analogous, with the roles of bottom and top 
elements exchanged. Thus, $j$ is a bottom position of~$\pi_n$. By parts \ref{it1} and \ref{it4} of 
the current lemma, $\Img(g)$, and therefore also $\Img(g')$ and $\Img(f')$, have 
no top position in the gap $(f(j-1),f(j))$, so $k$ is a bottom position and $k>j$.

The facts that $j\in \Img(f')$, $f'$ is greedy, and $j$ and $k$ are consecutive bottom positions in $\Img(g')$, imply that $k\in\Img(f')$. Thus, the two positions of 
$\lambda$ that are mapped to $j$ and $k$ by $f'$ form a bottom repetition in $\lambda$. This 
bottom repetition forms a gap in $f$ which is to the left of $j$, contradicting 
the definition of~$j$. This proves that $\Img(g)$ has no position in the leftmost gap of $f$.

Now we show that $\Img(g)$ has a top position in the tail of~$f$.
Since $\lambda$ is proper, the position $m-1$ is the rightmost top position in 
$\lambda$. By Part~\ref{it2} of the current lemma, we have $f'(m-1)>f(m-1)$. We 
claim that $f'(m-1)$ is in the tail of $f$: if not, then $(f(m-1),f(m))$ 
would be a $\ble{tb}$-gap in $f$ and $f'(m-1)$ would be a top position in this 
$\ble{tb}$-gap, contradicting Part~\ref{it4} of the current lemma. 
Therefore, $f'(m-1)$ is a top position in the intersection of $\Img(g')$ and the 
tail of~$f$. Finally, since $f$ has at least one gap, we have $f(m-1)\ge j-1$, implying $f'(m-1)>j$, and 
hence $f'(m-1)$ is in $\Img(g)$ as well.

\ref{it6} Let $(f(i),f(i+1))$ be the second gap of~$f$ from the left. Suppose that this 
is a $\ble{bt}$-gap; the other case is analogous, with the roles of bottom and top elements exchanged. By Lemma~\ref{lem-sf} we have $i\ge 4$, and since $f$ cannot have both a $\ble{tt}$-gap and a $\ble{bb}$-gap 
to the left of $f(i)$, the leftmost gap of $f$ is the $\ble{tt}$-gap $(f(i-2),f(i-1))$. 
For contradiction, 
suppose that $g$ contains a position $k$ in the $\ble{bt}$-gap $(f(i),f(i+1))$ of~$f$. By 
Part~\ref{it4}, we know that $k$ is a top position. 

By Part~\ref{it2}, we have $f'(i)>f(i)$, and since $\Img(g')$ has no bottom 
position in the $\ble{bt}$-gap $(f(i),f(i+1))$ of $f$, this implies $f'(i)>f(i+1)$.
Since $f'$ is 
greedy, it follows that $f'(i-1)\ge f(i+1)$, $f'(i-2)\ge k$, and $f'(i-3)\ge f(i)>f(i-2)$. 
We may now define a mapping $f^+$ as
\[
f^+(x)=\begin{cases} f'(x) \text{ if } x\ge i-3\\ 
f(x+2)\text{ if } x\le i-4.\end{cases}
\]
Since $1\notin \Img{f'}$, $f^+$ is an embedding of $\lambda$, clearly satisfying $f<_L f^+$. Since $f^+$ does not use the position $j$ from the first gap of $f$, $f^+$ is compatible with $g$, which contradicts $f$ being the rightmost in~$g$.

\ref{it7} Suppose first that $(f(i), f(i+1))$ is a $\ble{tt}$-gap and $(f(i+2), f(i+3))$ 
a $\ble{bb}$-gap. Since $\lambda$ is proper, we have $i\ge 3$.
Since $\lambda$ has at most one top repetition and at most one bottom repetition, 
Lemma~\ref{lem-sf} 
implies that the $\ble{tt}$-gap $(f(i),f(i+1))$ is the leftmost gap of~$f$. By part~\ref{it4}, 
$\Img(g)$, and consequently also $\Img(g')$ and $\Img(f')$, have no bottom position in the 
$\ble{bb}$-gap $(f(i+2), f(i+3))$. 
By Part~\ref{it2}, we have $f'(i+2)>f(i+2)$, and therefore $f'(i+2)\ge f(i+3)$. 
For contradiction, suppose that $\Img(g)$ contains at least two top positions $k_1<k_2$ in the $\ble{bb}$-gap of $f$.
The greediness of $f'$ implies 
$f'(i+1)\ge k_2$, $f'(i)\ge k_1$ and $f'(i-1)\ge f(i+2)$. 
We then define $f^+$ by
 \[
f^+(x)=\begin{cases} f'(x) \text{ if } x\ge i-1\\ 
f(x+2)\text{ if } x\le i-2.\end{cases}
\]
By the same reasoning as in Part~\ref{it6}, we get a contradiction with $f$ being the rightmost 
in~$g$.
\end{proof}

\subsection{Adding up all contributions} 

Let $\pi_n$ be fixed. Recall from~\eqref{eq-pp} and~\eqref{eq_Sf_singular} that 
\begin{equation}\label{eq-pp2}
 \mu(21,\pi_n)=\E(21,\pi_n) - \sum_{(\lambda,f)\in 
\cPP_n} 
(-1)^{|\lambda|}\NE(21,\lambda)\Sf,
\end{equation}
where $\Sf=\sum_{g\in\SEf}(-1)^{|g|}$.
We will now evaluate the sum on the right-hand side of \eqref{eq-pp2}. We 
will distinguish the proper pairs $(\lambda,f)$ depending on the number of 
repetitions of~$\lambda$ and the number of gaps of~$f$. For integers $a\le b$, 
we let $[a,b]$ denote the set $\{a,a+1,\dots,b\}$. When representing the 
structure of an embedding by its hyphen-letter notation, we underline the 
individual segments for added clarity, and we use the ellipsis 
`$\dots$' for segments of unknown length. We will use an auxiliary symbol 
`$\ble{\pgap}$' to denote a sequence of hyphens of arbitrary length, possibly empty. In particular, \texttt{-*-} represents a sequence of hyphens of length at least $2$, and if `$\pgap$' is adjacent to a segment from both left and right, the two segments may possibly form a single segment. We say that such a potentially empty sequence of hyphens represents a \emph{potential gap}.

\paragraph*{Case A: $\lambda$ has no repetitions} 
Then $f$ has no gaps, by 
Lemma~\ref{lem-sf}. In the hyphen-letter notation, we have
\[
f=\ble{\seg{lbtb\dots tbtb} \bu\pgap\bu}.
\] 
Let $\cPP_A$ be the set of proper pairs $(\lambda,f)$ of this form; similarly, 
$\cPP_B, \cPP_C,\allowbreak \cPP_D$ and $\cPP_E$ will be the sets of proper 
pairs to 
be considered in subsequent cases. 

Fix a proper pair $(\lambda,f)\in \cPP_A$.
We claim that $\SEf$ contains exactly one embedding $g_A$, determined by
$\Img(g_A)=\Img(f)\cup\{2n+2\}$; that is, 
\[
g_A=\ble{\seg{lbtb\dots tbtb} \bu\pgap \seg{r}}.
\] 
It is easy to see that $f$ is the rightmost embedding in $g_A$ (recall that by 
Lemma~\ref{lem-proper}, no embedding of $\lambda$ may contain both $1$ and $2n+2$ in its image, and 
in particular, 
there is no greedy embedding of $\lambda$ in $g_A$). Since $|g_A|$ is odd, the contribution of 
$g_A$ to $S_f$ is $(-1)^{|g_A|}=-1$.

Now assume that $g\in \SEf$. We have $2n+2 \in \Img(g)$ by 
Lemma~\ref{lem-singular} \ref{it3}. By Lemma~\ref{lem-singular} \ref{it4}, $\Img(g)$ has no 
bottom position in the tail of~$f$. We claim that $\Img(g)$ has no top position 
in the tail of $f$ either. Indeed, if $\Img(g)$ contained a top 
position $k$ in the tail of $f$, then $g$ would be compatible with a greedy embedding $f^+$ of $\lambda$ satisfying 
$\Img(f^+)=\left(\Img(f)\cup\{k,2n+2\}\right)\setminus\{1,2\}$, and this would contradict 
Lemma~\ref{lem-singular}~\ref{it3}. Therefore, 
$\Img(g)=\Img(f)\cup\{2n+2\}$.

To compute $\sum_{(\lambda,f)\in\cPP_A}(-1)^{|\lambda|}\NE(21,\lambda)\Sf$, we 
reason as 
follows: to every triple $(i_1,i_2,i_3)$ with $1\le i_1\le i_2\le i_3\le n$ we 
associate a proper $\lambda$ of size $2i_3$ with no repetitions, and a 
normal embedding $h$ of 21 into $\lambda$ with $\Img(h)=\{2i_1-1,2i_2\}$. As 
there are $\binom{n+2}{3}$ triples $(i_1,i_2,i_3)$ of this form, and $\Sf=-1$ 
for all $(\lambda,f)\in\cPP_A$, we get
\[
 \sum_{(\lambda,f)\in\cPP_A}(-1)^{|\lambda|}\NE(21,\lambda)\Sf = 
-\binom{n+2}{3}.
\]

\paragraph*{Case B: $\lambda$ has a bottom repetition but no top repetition} 
Then, by Lemma~\ref{lem-sf}, $f$ has a $\ble{bb}$-gap, and possibly also a $\ble{tb}$-gap 
immediately 
following it; that is, 
\[
f =\ble{\seg{lb\dots tb}\bu \pgap \seg{bt}\pgap \seg{btb\dots tb}\bu\pgap \bu}\\
\]
with the second segment of length exactly $2$ and the third of length at 
least~$1$, and these two are possibly combined into a single segment.

By Lemma~\ref{lem-singular} \ref{it3} and \ref{it5}, if $g\in\SEf$, then $\Img(g)$ contains $2n+2$ 
as well as at least one 
top position in the tail of~$f$. On the other hand, by Lemma~\ref{lem-singular} \ref{it5} and 
\ref{it6}, $\Img(g)$ has no position in the gaps of~$f$.


Conversely, we claim that if $\Img(g)=\Img(f)\cup T\cup\{2n+2\}$ where $T$ 
is a nonempty set of top positions in the tail of~$f$, then $g\in \SEf$.
Clearly $g$ is compatible with $f$ since $\Img(f)\subset\Img(g)$. To show that $g\in \sEf$, we 
observe that $f$ is almost greedy in $g$ and that there is no greedy embedding of 
$\lambda$ in $g$, since every embedding of $\lambda$ compatible with $g$ must coincide with $f$ on all top positions of $\lambda$ before the repetition. It remains to show that $g$ is singular. If $j$ is the leftmost position in the leftmost gap of $f$, $g'=\Delta_j(g)$, $k$ is the second position in the second segment of $f$, and $l$ is the rightmost position of $T$, then the embedding $f^+$ of $\lambda$ with image $\Img(f)\setminus\{1,2,k\}\cup\{j,l,2n+2\}$ is greedy in $g'$ and satisfies $f<_L f^+$.

We now compute the value of~$S_f$.
We can apply the involution $\Delta_{2n+1}$ to cancel out the contribution of 
those embeddings $g\in\SEf$ that contain at least one top position in the tail 
different from~$2n+1$. This leaves exactly one embedding in $\SEf$ that is not 
cancelled, namely,
\[
 g_B =\ble{\seg{lb\dots tb}\bu \pgap \seg{bt}\pgap \seg{btb\dots tb}\pgap \seg{tr}}.
\]
Since $|g_B|$ is odd, the contribution of $g_B$ to $S_f$ is $(-1)^{|g_B|}=-1$ and hence $S_f=-1$.

Note that a normal embedding of $21$ into $\lambda$ must map the second element 
of $21$ to the second element of the bottom repetition of~$\lambda$. To compute 
$\sum_{(\lambda,f)\in\cPP_B}(-1)^{|\lambda|} \NE(21,\lambda)\Sf$, we encode the contributions 
to this sum as quintuples $(i_1,i_2,\dots,i_5)$ with $1\le i_1\le 
i_2<i_3<i_4\le i_5\le n$, corresponding to the embedding $f$ with segments 
$[1, 2i_2]$,  $\{2i_3,2i_3+1\}$, and $[2i_4,2i_5]$ (the latter two segments 
possibly merged into a single one), and the normal embedding $h$ of $21$ into 
$\lambda$ specified by $\Img(fh)=\{2i_1-1,2i_3\}$, where $fh$ is the embedding of $21$ to $\pi_n$ 
that is a composition of $h$ and $f$. Since $|\lambda|$ is odd, we have
\[
 \sum_{(\lambda,f)\in\cPP_B} (-1)^{|\lambda|}\NE(21,\lambda)\Sf=\binom{n+2}{5}.
\]

\paragraph*{Case C: $\lambda$ has a top repetition and no bottom repetition} The proper 
permutations $\lambda$ of this form are precisely the reverse-complements of the permutations 
considered in Case B. From this, we may deduce that the contributions of the two cases are equal. 
To see this, let $\cP_B$ denote the set of all the proper permutations with a bottom repetition and 
no top repetition, and $\cP_C$ the set of all the proper permutations with a top repetition and no 
bottom repetition. We then obtain
\begin{align*}
\sum_{(\lambda,f)\in\cPP_C}\!(-1)^{|\lambda|}\NE(21,\lambda)\Sf&=\sum_{\lambda\in\cP_C} 
(-1)^{|\lambda|}\NE(21,\lambda)S_\lambda\\
&=\sum_{\lambda\in\cP_B} (-1)^{|\lambda|}\NE(21,\lambda)S_\lambda 
&\text{(Lemma~\ref{lem-properrc})}\\
&=\sum_{(\lambda,f)\in\cPP_B}\! (-1)^{|\lambda|}\NE(21,\lambda)\Sf\\&=\binom{n+2}{5}.
\end{align*}

\paragraph*{Case D: $\lambda$ has two repetitions, and the top repetition is not 
adjacent to the bottom one} We then have $\NE(21,\lambda)=1$, 
and by Lemma~\ref{lem-sf}, $f$ has the form
\[
f =\seg{lbt\dots bt}\bu \pgap \seg{tb}\pgap 
\seg{tb\dots tb} \bu\pgap \seg{bt} \pgap 
\seg{bt\dots tb}\bu\pgap\bu,
\]
with the second and the fourth segments of length $2$, either of them possibly 
combined with the following segment into a segment of length at least~$3$. 

Fix $g\in\SEf$. By Lemma~\ref{lem-singular} \ref{it3}, \ref{it5} and \ref{it6}, $\Img(g)$ contains 
$2n+2$ and at least one top position in the tail of~$f$, but it has no position in the 
$\ble{tt}$-gap or the $\ble{bt}$-gap of~$f$. Moreover, By Lemma~\ref{lem-singular} \ref{it4}, 
$\Img(g)$ has no bottom position in the $\ble{bb}$-gap of $f$ and no top position in the 
$\ble{tb}$-gap of~$f$. We conclude 
that $\Img(g)=\Img(f)\cup \Tbb\cup \Btb \cup T\cup \{2n+2\}$ where $\Tbb$ is a 
set of top positions in the $\ble{bb}$-gap of $f$, $\Btb$ is a set of bottom 
positions in the $\ble{tb}$-gap of $f$, and $T$ is a nonempty set of top positions in the tail 
of~$f$.

Moreover, at least one of $\Tbb$ and $\Btb$ must be nonempty, otherwise the mapping 
$f'$ defined as in Lemma~\ref{lem-singular} would have to map the top elements in the third segment 
of $f$ to the same positions as $f$, contradicting Lemma~\ref{lem-singular}~\ref{it2}. 

On the other hand, it cannot happen that $\Tbb$ and $\Btb$ are both nonempty, because in this case $g$ would admit a greedy embedding of~$\lambda$, contradicting 
Lemma~\ref{lem-singular}~\ref{it3}; we illustrate this in the following example, where $f^+$ is the greedy 
embedding of $\lambda$ into~$g$:
\begin{align*}
f &=\ble{\seg{lbtbtbt}\bu\bu\bu \seg{tb}\bu\bu \seg{tbtb} \bu\bu\bu \seg{bt} \bu\bu
\seg{btb}\bu\bu\bu\bu},\\
g &=\ble{\seg{lbtbtbt}\bu\bu\bu \seg{tb}\bu\bu \seg{tbtbt} \bu\bu \seg{btb}\bu
\seg{btb}\bu\bu \seg{tr}}, \\
f^+ &=\ble{\bu\bu\seg{tbtb}\bu \bu\bu\bu \seg{tb}\bu\bu \seg{t}\bu\seg{tbt}\bu \bu \seg{btb}\bu
\seg{btb}\bu\bu \seg{tr}}.
\end{align*}
More explicitly, the greedy embedding $f^+$ maps the rightmost position of $\lambda$ to $2n+2$, the 
second rightmost position to the rightmost position of $T$, the remaining positions from the fourth 
and fifth segment of $f$ to positions of the fifth segment of $f$, the left 
element of the bottom repetition of $\lambda$ to the rightmost position in $\Btb$,  
the remaining elements of the second and third segment of $f$ to the positions 
of the fourth segment of $f$, the rightmost position of $\Tbb$, and the positions of the third 
segment of $f$ except the leftmost two. The left element of the top repetition of $\lambda$ then 
gets mapped to the leftmost position of the third segment, and the remaining elements of the first 
segment of $f$ are shifted to the right greedily, freeing the leftmost two positions of~$\pi_n$. By 
Lemma~\ref{lem-proper}, $f^+$ is an embedding of $\lambda$, contradicting $g\in\SEf$.

In the above example, we focus on the situation when $\Tbb$ and $\Btb$ are both singleton sets, 
that is, when $\Img(g)$ is as small as possible with respect to a given~$f$. This does not lose any 
generality, since adding more positions to $\Img(g)$ would not change the fact that $g$ is 
compatible with the embedding $f^+$ and therefore $f$ is not rightmost in~$g$. We also note that the 
only way in which a fully general $f$ can deviate from the specific example illustrated above is in 
the lengths of the gaps and the tail, and in the length of the first, third and fifth segment. The 
gaps and segments of $g$ and of $f^+$ can be adjusted in an obvious way to match a given~$f$. The 
remarks in this paragraph apply to our future examples in this section as well, and we will refrain 
from repeating them explicitly.

We have concluded that for any $g\in\SEf$, exactly one of the two sets $\Tbb$ and $\Btb$ is empty.
Conversely, it is straightforward to verify that if $g$ is an embedding whose image has the form 
$\Img(f)\cup \Tbb\cup \Btb \cup T\cup \{2n+2\}$ as above, with exactly one of $\Tbb$ and $\Btb$ 
being empty, then $g\in\SEf$; see the following examples: 

\begin{align*}
f &=\ble{\seg{lbtbtbt}\bu \bu\bu \seg{tb}\bu\bu \seg{tbtb} \bu\bu\bu \seg{bt} \bu\bu\seg{btb}\bu\bu\bu\bu},\\
g_1 &=\ble{\seg{lbtbtbt}\bu \bu\bu \seg{tb}\bu\bu \seg{tbtbt} \bu\bu \seg{bt}\bu\bu\seg{btb}\bu\bu \seg{tr}},\\
f'_1 &=\ble{\bu\bu\seg{tbtbtb}\bu\bu \seg{t}\bu \bu\bu\seg{tbtbt}\bu\bu \seg{b}\bu\bu\bu\seg{btb}\bu\bu \seg{tr}},\\
%
g_2 &=\ble{\seg{lbtbtbt}\bu \bu \bu \seg{tb}\bu \bu \seg{tbtb} \bu\bu\bu \seg{btb}\bu\seg{btb}\bu\bu \seg{tr}}, \\
f'_2 &=\ble{\bu\bu\seg{tbtbtb}\bu \bu \seg{t}\bu \bu\bu\seg{tbt}\bu\bu\bu\bu \seg{btb}\bu\seg{btb}\bu\bu \seg{tr}}.
\end{align*}
As the examples show, if $g_1$ is an embedding satisfying $\Img(g_1)=\Img(f)\cup \Tbb 
\cup T\cup \{2n+2\}$ for nonempty $\Tbb$ and $T$, and $j$ is the leftmost position not in $\Img(f)$, then there is a greedy embedding $f'_1$ of $\lambda$ in $g'_1=\Delta_j(g_1)$. Likewise, for $g_2$ with $\Img(g_2)=\Img(f)\cup \Btb \cup T\cup \{2n+2\}$ with $\Btb$ and $T$ nonempty, there is a 
greedy embedding $f'_2$ of $\lambda$ in $g'_2=\Delta_j(g_2)$. This shows that the 
embeddings $g_1$ and $g_2$ belong to~$\SEf$. 

Note that $\Btb$ can be nonempty only when $f$ has a $\ble{tb}$-gap.

As in Case B, we can apply the involution $\Delta_{2n+1}$ to cancel out the contributions of all 
$g\in\SEf$ except those for which $T=\{2n+1\}$. By an analogous argument, we cancel out all 
$g\in\SEf$ except those for which $\Tbb$ is either empty or a singleton set containing the leftmost 
element of the $\ble{bb}$-gap of~$f$, and $\Btb$ is either empty or a singleton set containing the 
leftmost element of the $\ble{tb}$-gap of~$f$. After these cancellations, the contribution of 
$\SEf$ restricts to just two embeddings $g_1$ and $g_2$ shown in the two examples,
where $g_2$ is only applicable if $f$ has a $\ble{tb}$-gap.
Since both $|g_1|$ and $|g_2|$ are odd, the contribution of each of $g_1$ and $g_2$ to $S_f$ is 
$-1$.

%

The embeddings 
of the form $g_1$ can be encoded by 7-tuples 
$1\le i_1<i_2<i_3<i_4<i_5<i_6\le i_7\le n$ where $\Img(f)=[1,2i_1+1]\cup\{2i_2+1,2i_2+2\}\cup[2i_3+1,2i_4]\cup\{2i_5,2i_5+1\}\cup[2i_6,2i_7]$. The embeddings $g_2$ can be encoded in the same way, only now 
we have the extra condition that $i_6>i_5+1$. Therefore, there are $\binom{n+1}{7}$ embeddings of the form $g_1$ and 
$\binom{n}{7}$ embeddings of the form $g_2$, so the total contribution from Case D is 
\[
\sum_{(\lambda,f)\in\cPP_D} 
(-1)^{|\lambda|}\NE(21,\lambda)\Sf=-\binom{n+1}{7}-\binom{n}{7}.
\]

\paragraph*{Case E: $\lambda$ has two repetitions, and they are adjacent to each 
other} Again, we have $\NE(21,\lambda)=1$. By Lemma~\ref{lem-sf}, the form of $f$ is
\[
f =\seg{lbt\dots bt}\bu \pgap\seg{tb}\bu \pgap  
\seg{b}\pgap\seg{t} \pgap\seg{btb\dots tb}  \bu\pgap\bu,
\]
where the third and fourth segment, as well as the fourth and fifth segment, are 
separated by a potential gap; that is, any of these two pairs of consecutive 
segments may in fact be merged into a single segment.

Fix $g\in\SEf$. By Lemma~\ref{lem-singular}~\ref{it4} and \ref{it5}, we have $\Img(g)=\Img(f)\cup\Tbb\cup \Tbt \cup \Btb\cup T\cup\{2n+2\}$ where $\Tbb$ is a set of top 
positions in the $\ble{bb}$ gap of $f$, $\Tbt$ and~$\Btb$ are defined analogously, and $T$ is a 
nonempty set of top positions in the tail of~$f$. In addition, Lemma~\ref{lem-singular}~\ref{it7} implies $|\Tbb|\le 1$.

First we assume that $\Btb$ is nonempty; this is of course only possible when $f$ has a 
$\ble{tb}$-gap. Then $\Tbb$ and $\Tbt$ are both empty, otherwise $g$ would admit a greedy embedding 
of~$\lambda$, as illustrated by the following examples:
\begin{align*}
f &=\seg{lbtbt}\bu\bu\bu\seg{tb}\bu\bu\bu \seg{b}\bu\bu\seg{t} \bu\bu\seg{btb} \bu\bu\bu\bu,\\
g_1  &=\seg{lbtbt}\bu\bu\bu\seg{tb}\bu\bu\bu \seg{bt}\bu\seg{tb} \bu\seg{btb}\bu \bu\seg{tr},\\
f'_1 &=\bu\bu\seg{tb}\bu\bu\bu\bu\seg{t}\bu\bu\bu\bu \seg{bt}\bu\seg{tb}\bu\seg{btb}\bu\bu\seg{tr},\\
g_2  &=\seg{lbtbt}\bu\bu\bu\seg{tbt}\bu\bu \seg{b}\bu\bu\seg{tb} \bu\seg{btb}\bu \bu\seg{tr},\\
f'_2 &=\bu\bu\seg{tb}\bu\bu\bu\bu\seg{tbt}\bu\bu\bu\bu\bu\seg{tb}\bu\seg{btb}\bu\bu\seg{tr}.
\end{align*}
As the first example illustrates, if $g_1$ is an embedding with $\Img(g_1)=\Img(f)\cup \Tbt\cup 
\Btb\cup T\cup\{2n+2\}$ with $\Tbt$, $\Btb$ and $T$ all nonempty, then $g_1$ admits the greedy 
embedding $f'_1$ of $\lambda$, and in particular $g_1\not\in\SEf$. Likewise, an embedding $g_2$ 
with $\Img(g_2)=\Img(f)\cup \Btb\cup\Tbb\cup T\cup\{2n+2\}$ and $\Btb$, $\Tbb$ and $T$ all nonempty 
admits the greedy embedding $f'_2$ of $\lambda$ and hence does not belong to~$\SEf$.

Thus $\Img(g)=\Img(f)\cup \Btb\cup T\cup\{2n+2\}$ with $\Btb$ and $T$ both nonempty. Conversely, every $g$ 
with $\Img(g)$ of this form belongs to~$\SEf$; see the following example, which shows that for 
$j=\min [2n+2]\setminus\Img(f)$, the embedding $g'=\Delta_j(g)$ admits a greedy embedding $f'$ of 
$\lambda$:
\begin{align*}
f &=\seg{lbtbt}\bu\bu\bu\seg{tb}\bu\bu\bu \seg{b}\bu\bu\seg{t}\bu\bu\bu\bu\seg{btb} \bu\bu\bu\bu,\\
g &=\seg{lbtbt}\bu\bu\bu\seg{tb}\bu\bu\bu \seg{b}\bu\bu\seg{tb}\bu\bu\bu\seg{btb}\bu 
\bu\seg{tr},\\
f' &=\bu\bu\seg{tbtb}\bu\bu\seg{t}\bu\bu\bu\bu\bu\bu\bu\seg{tb}\bu\bu\bu\seg{btb}\bu\bu\seg{tr}.
\end{align*}

Applying analogous cancellations as in Case B and Case D, we cancel out all $g$ of this form except 
those with $T=\{2n+1\}$ and $\Btb$ a singleton set containing the leftmost position in the 
$\ble{tb}$-gap of~$f$. Thus the contribution of this type of $g$ to $\SEf$ reduces to the following 
embedding $g$, shown along with 
$f$ for clarity:
\begin{align*}
f& =\seg{lbt\dots bt}\bu \pgap \seg{tb}\bu \pgap
\seg{b}
\pgap
\seg{t}\bu
\pgap\bu
\seg{btb\dots tb}
\bu\pgap\bu,\\
g& =\seg{lbt\dots bt}\bu \pgap \seg{tb}\bu \pgap
\seg{b}
\pgap
\seg{tb}
\bu \pgap
\seg{btb\dots tb}
\pgap\seg{tr}.
\end{align*}
Since $|g|$ is odd, the contribution of $g$ to $S_f$ is $-1$.
Every $g$ of this form can be encoded by a 6-tuple $1\le i_1<i_2<i_3< 
i_4<i_5< i_6\le n$ where $\Img(f)=[1,2i_1+1]\cup\{2i_2+1,2i_2+2, 2i_3+2, 2i_4+1\}\cup[2i_5+2,2i_6]$. Therefore, the number of these embeddings $g$ is $\binom{n}{6}$.

Now we assume that $\Btb$ is empty.  
We claim that $\Tbb$ is nonempty and thus $|\Tbb|=1$: if $\Tbb$ were empty, every embedding of 
$\lambda$ compatible with $g'$ would coincide with $f$ on the top positions in the first two 
segments of $f$; in particular, there would be no greedy embedding of $\lambda$ in $g'$.

We conclude that $\Img(g)=\Img(f)\cup\{k\}\cup \Tbt\cup T\cup\{2n+2\}$ where 
$k$ is a top position in the $\ble{bb}$-gap of $f$, $\Tbt$ is a possibly empty set of top positions 
in the $\ble{bt}$-gap of $f$, and $T$ is a nonempty set of top positions in the tail of~$f$. 

Conversely, every such $g$ belongs to $\SEf$; see the following example, illustrating the greedy 
embedding $f'$ into $g'=\Delta_j(g)$, where $j=\min[2n+2]\setminus\Img(f)$:
\begin{align*}
f&=\seg{lbtbt}\bu\bu\bu\seg{tb}\bu\bu\bu\seg{b}\bu\bu\bu\bu\seg{t}\bu\bu\seg{btb} \bu\bu\bu\bu,\\
g&=\seg{lbtbt}\bu\bu\bu\seg{tbt}\bu\bu  \seg{b}\bu\bu\bu\bu\seg{t}\bu\bu\seg{btb}\bu\bu\seg{tr},\\
f' &=\bu\bu\seg{tbtb}\bu\bu\seg{t}\bu\seg{t}\bu\bu\seg{b}\bu\bu\bu \bu\bu\bu\bu\seg{btb}\bu\bu\seg{tr}.
\end{align*}

The contributions of these embeddings sum to zero whenever $f$ has a $\ble{bt}$-gap, since for 
any top position $i$ in the $\ble{bt}$-gap, the $i$-switch $\Delta_i$ is a parity-exchanging 
involution on these embeddings. If $f$ has no $\ble{bt}$-gap, usual cancellations restrict the  
contributions of this type of $g$ to the case $T=\{2n+1\}$, corresponding to the following $f$ 
and~$g$:
\begin{align*}
f& =\seg{lbt\dots bt}\bu \pgap \seg{tb}
\bu \pgap \bu
\seg{bt}
\pgap
\seg{btb\dots tb}
\bu\pgap\bu,\\
g& =\seg{lbt\dots bt}\bu \pgap \seg{tb}
\pgap \seg{t} \pgap
\seg{bt}
\pgap
\seg{btb\dots tb}
\pgap\seg{tr}.
\end{align*}
Since $|g|$ is odd, the contribution of $g$ to $S_f$ is $-1$.
Every $g$ of this form can be encoded by a 6-tuple $1\le i_1<i_2<i_3< 
i_4<i_5\le i_6\le n$ where $\Img(g)=[1,2i_1+1]\cup\{2i_2+1,2i_2+2,2i_3+1,2i_4,2i_4+1\}\cup[2i_5,2i_6]\cup\{2n+1,2n+2\}$.
Therefore, the number of these embeddings $g$ is $\binom{n+1}{6}$.

Adding the contributions of the two types of embeddings considered in Case~E, we get 
\[
\sum_{(\lambda,f)\in\cPP_E} 
(-1)^{|\lambda|}\NE(21,\lambda)\Sf=-\binom{n}{6}-\binom{n+1}{6}.
\]

\paragraph*{Final count}
Adding the contributions of Case D and E, we get $-\binom{n+1}{7}- \binom{n}{7} 
- \binom{n+1}{6} - \binom{n}{6}= -\binom{n+2}{7}-\binom{n+1}{7}$. 

Note that $\E(21,\pi_n)= \binom{n}{2}+2n$: indeed, all the embeddings of $21$ into $\pi_n$ are 
normal, and among all such embeddings $f$, there are exactly $n$ choices for which $f(1)=1$ and 
$f(2)$ is a bottom position, $n$ choices for which $f(1)$ is a top position and $f(2)=2n+2$, and 
$\binom{n}{2}$ choices for which $f(1)$ is a top position, $f(2)$ a bottom position and $f(1)<f(2)$.

Summing all the contributions together, we get 
\begin{align*}
\mu(1,\pi_n)&=-\mu(21,\pi_n)= -\E(21,\pi_n) + \sum_{(\lambda,f)\in\cPP_n} 
(-1)^{|\lambda|}\NE(21,\lambda)\Sf\\
&= -\binom{n+2}{7}-\binom{n+1}{7} +2\binom{n+2}{5} -\binom{n+2}{3} -\binom{n}{2}-2n,
\end{align*}
and Theorem~\ref{thm-cross} is proved.

\section{Further directions and open problems}

Determining the fastest possible growth of $|\mu(1,\pi)|$ as a function of~$|\pi|$ is still widely open. Defining
\[ 
f(n) = \max \{|\mu(1,\pi)|; |\pi|=n\},
\]
Theorem~\ref{thm-cross} gives an asymptotic lower bound $f(n) \ge \Omega(n^7)$.
We believe this is just a first step towards proving much better lower bounds on $f(n)$. The main obstacle here is our inability to compute or even estimate $|\mu(1,\pi)|$ for a general~$\pi$. 

\begin{figure}
  \centerline{\includegraphics[width=0.8\linewidth]{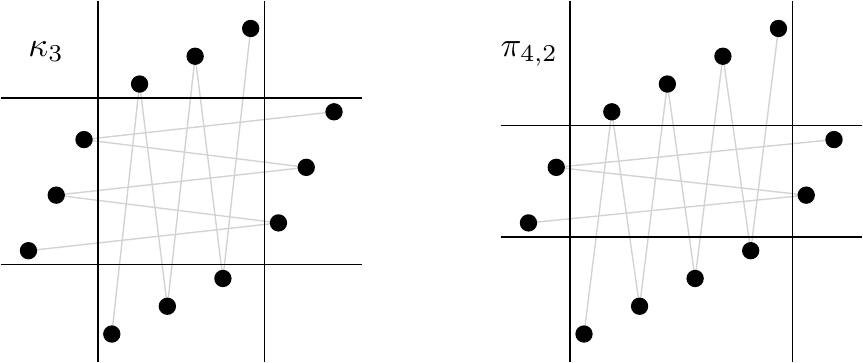}}
  \caption{The permutations $\kappa_3$ (left) and $\pi_{4,2}$ (right).}
  \label{fig-kappa}
\end{figure}

Our computational experiments suggest that $\mu$ might grow exponentially fast even for permutations of seemingly simple structure. Let 
$\kappa_n\in\cS_{4n}$ be a permutation defined as
\[
\kappa_n=n+1,n+3,\dots,3n-1,1,3n+1,2,3n+2,\dots,n,4n,n+2,n+4,\dots,3n;
\]
see Figure~\ref{fig-kappa}. Note that $\kappa_n$ is a $321$-free permutation 
that can be split into four `quadrants', each consisting of an increasing 
subsequence of length~$n$. 

\begin{conjecture}\label{conj-exp}
The absolute value of $\mu(1,\kappa_n)$ is exponential in $n$.
\end{conjecture}

The permutation $\pi_n$ is a subpermutation of $\kappa_n$ for which we were able to compute $\mu(1,\pi_n)$ precisely due to a relatively simple structure of the interval $[1,\pi_n]$. For $k\le n$, let $\pi_{n,k}$ be the subpermutation of $\kappa_n$ induced by the values $n+1,n+3,\dots,n+2k-1,1,3n+1,2,3n+2,\dots,n,4n,n+2,n+4,\dots,n+2k$ in~$\kappa_n$; see Figure~\ref{fig-kappa}. In particular, $\pi_{n,1}=\pi_n$.
Our intuition and some preliminary results support the following conjecture.

\begin{conjecture}\label{conj_poly}
For every fixed $k\ge 1$, the absolute value of $\mu(1,\pi_{n,k})$ grows as $\Theta(n^{k^2+6k})$.
\end{conjecture}

From Philip Hall's Theorem (Fact~\ref{fac-hall}), we see that $|\mu(1,\pi)|$ is 
bounded from above by the number of chains from $1$ to $\pi$ in the interval $[1,\pi]$, which is further bounded from above by the number of chains from $\emptyset$ to $[n]$ in the poset $(\mathcal{P}([n]),\subseteq)$ of all subsets of $[n]$; see A000670 in OEIS~\cite{sloane}. This gives the rough upper bound
\[
f(n) \le \left(\frac{1}{\log_e 2} \right)^n \cdot n! < (1.443)^n \cdot n!
\]
for large $n$.

This bound can be further improved using 
Ziegler's result~\cite[Lemma 4.6]{ZieglerMax}, which states that the M\"obius function of an interval $[x,y]$ in any locally finite poset is bounded from above by the number of maximal chains from $x$ to~$y$. Again, by counting maximal chains in $(\mathcal{P}([n]),\subseteq)$, we get the upper bound 
\[
f(n)\le n!.
\]
This is still far even from the exponential lower bound proposed in Conjecture~\ref{conj-exp}.

\begin{problem} 
Is $f(n)\le 2^{O(n)}$?
\end{problem}

We note without giving further details that the number of maximal chains from $1$ to $\pi$ can grow as $2^{\Omega(n \log n)}$ for some permutations $\pi$ of size $n$, including our permutation $\pi_n$. 

\paragraph*{Hereditary classes}
Suppose that $\pi$ is restricted to a given proper down-set $\mathcal{C}$ of $(\cS, \le)$, that is, to a 
hereditary permutation class. Determining whether the values of $\mu(1,\pi)$ are bounded by a constant for $\pi \in \mathcal{C}$ is also an interesting problem.
Burstein et al.~\cite{BJJS} show that $\mu(1,\pi)$ 
is bounded on the class of the so-called separable permutations, while 
Smith~\cite{Smith_one,Smith_descents} shows that it is unbounded on permutations 
with at most one descent. These results suggest that the growth of $\mu(1,\pi)$ might
depend on the so-called simple permutations in the class, where a permutation is \emph{simple} if 
it maps no nontrivial interval of consecutive positions to an interval of consecutive 
values, such as the interval 645 in 71645283.

\begin{problem}\label{prob}
On which hereditary permutation classes is $\mu(1,\pi)$ bounded? Is it bounded 
on every class with finitely many simple permutations? Is it unbounded on every 
class with infinitely many simple permutations?
\end{problem}

Very recently, as this paper was undergoing review, Marchant has published a 
preprint~\cite{Marchant} with a construction of a sequence of permutation whose principal M\"obius 
function is exponential in their size. This result, if confirmed, would greatly improve upon our 
polynomial lower bound on $f(n)$, and also answer in the negative the first question in 
Problem~\ref{prob}, since Marchant's construction is based on permutations from a class with only 
finitely many simple permutations.

\section*{Acknowledgements} We wish to thank Vojta Kalu\v za for useful 
discussions.
 
\footnotesize
\bibliographystyle{vlastni}
\bibliography{perms}

\begin{thebibliography}{10}
\expandafter\ifx\csname url\endcsname\relax
  \def\url#1{{\tt #1}}\fi
\expandafter\ifx\csname urlprefix\endcsname\relax\def\urlprefix{URL }\fi

\bibitem{BjornerSubword}
A.~Bj\"orner, The {M}\"obius function of subword order, {\em Invariant theory
  and tableaux ({M}inneapolis, {MN}, 1988)\/}, vol.~19 of {\em IMA Vol. Math.
  Appl.\/}, Springer, New York (1990) 118--124.

\bibitem{BjornerFactor}
A.~Bj\"orner, The {M}\"obius function of factor order, {\em Theoret. Comput.
  Sci.\/} {\bf 117}(1-2) (1993), 91--98.

\bibitem{BBL}
P.~Bose, J.~F. Buss and A.~Lubiw, Pattern matching for permutations, {\em
  Inform. Process. Lett.\/} {\bf 65}(5) (1998), 277--283.

\bibitem{Zeros}
R.~Brignall, V.~Jel{\'i}nek, J.~Kyn{\v c}l and D.~Marchant, Zeros of the
  {M}{\"o}bius function of permutations (2018), to appear in Mathematika.

\bibitem{BM}
R.~Brignall and D.~Marchant, The {M}{\"o}bius function of permutations with an
  indecomposable lower bound, {\em Discrete Math.\/} {\bf 341}(5) (2018),
  1380--1391.

\bibitem{BJJS}
A.~Burstein, V.~Jel\'{\i}nek, E.~Jel\'{\i}nkov\'a and E.~Steingr\'{\i}msson,
  The {M}\"obius function of separable and decomposable permutations, {\em J.
  Combin. Theory Ser. A\/} {\bf 118}(8) (2011), 2346--2364.

\bibitem{EdelmanSimion}
P.~H. Edelman and R.~Simion, Chains in the lattice of noncrossing partitions,
  {\em Discrete Math.\/} {\bf 126}(1-3) (1994), 107--119.

\bibitem{EhrenborgReaddy}
R.~Ehrenborg and M.~A. Readdy, The {M}\"obius function of partitions with
  restricted block sizes, {\em Adv. in Appl. Math.\/} {\bf 39}(3) (2007),
  283--292.

\bibitem{goyt}
A.~M. Goyt, The {M}\"obius function of a restricted composition poset, {\em Ars
  Combin.\/} {\bf 126} (2016), 177--193.

\bibitem{Marchant}
D.~Marchant, 2413-balloon permutations and the growth of the {M}{\"o}bius
  function, arXiv:1812.05064 (2018),
  \urlprefix\url{https://arxiv.org/abs/1812.05064}.

\bibitem{McNamaraSagan}
P.~R.~W. McNamara and B.~E. Sagan, The {M}\"obius function of generalized
  subword order, {\em Adv. Math.\/} {\bf 229}(5) (2012), 2741--2766.

\bibitem{McNamaraSt15}
P.~R.~W. McNamara and E.~Steingr\'{\i}msson, On the topology of the permutation
  pattern poset, {\em J. Combin. Theory Ser. A\/} {\bf 134} (2015), 1--35.

\bibitem{SaganVatter}
B.~E. Sagan and V.~Vatter, The {M}\"obius function of a composition poset, {\em
  J. Algebraic Combin.\/} {\bf 24}(2) (2006), 117--136.

\bibitem{sloane}
N.~J.~A. Sloane, {The {O}n-{L}ine {E}ncyclopedia of {I}nteger {S}equences,
  published electronically at \url{http://oeis.org/}}.

\bibitem{Smith_one}
J.~P. Smith, On the {M}\"obius function of permutations with one descent, {\em
  Electron. J. Combin.\/} {\bf 21}(2) (2014), Paper 2.11, 19 pp.

\bibitem{Smith_descents}
J.~P. Smith, Intervals of permutations with a fixed number of descents are
  shellable, {\em Discrete Math.\/} {\bf 339}(1) (2016), 118--126.

\bibitem{Smith_formula}
J.~P. Smith, A formula for the {M}\"obius function of the permutation poset
  based on a topological decomposition, {\em Adv. in Appl. Math.\/} {\bf 91}
  (2017), 98--114.

\bibitem{Stan}
R.~P. Stanley, {\em Enumerative combinatorics. {V}olume 1\/}, vol.~49 of {\em
  Cambridge Studies in Advanced Mathematics\/}, second ed., Cambridge
  University Press, Cambridge (2012), ISBN 978-1-107-60262-5.

\bibitem{ST}
E.~Steingr\'{\i}msson and B.~E. Tenner, The {M}\"obius function of the
  permutation pattern poset, {\em J. Comb.\/} {\bf 1}(1) (2010), 39--52.

\bibitem{Wachs}
M.~L. Wachs, Poset topology: tools and applications, {\em Geometric
  combinatorics\/}, vol.~13 of {\em IAS/Park City Math. Ser.\/}, Amer. Math.
  Soc., Providence, RI (2007) 497--615.

\bibitem{Wilf}
H.~S. Wilf, The patterns of permutations, {\em Discrete Math.\/} {\bf 257}(2-3)
  (2002), 575--583.

\bibitem{Ziegler}
G.~M. Ziegler, On the poset of partitions of an integer, {\em J. Combin. Theory
  Ser. A\/} {\bf 42}(2) (1986), 215--222.

\bibitem{ZieglerMax}
G.~M. Ziegler, Posets with maximal {M}\"obius function, {\em J. Combin. Theory
  Ser. A\/} {\bf 56}(2) (1991), 203--222.

\end{thebibliography}

\end{document}